\documentclass[11pt,a4paper]{amsart}
\usepackage[toc,page]{appendix}
\usepackage{a4}
\usepackage[active]{srcltx}
\usepackage{amsmath,bbm,esint}
\usepackage{amssymb}
\usepackage{accents,calc}
\usepackage{subfig}
\usepackage{caption}

\usepackage{mwe}
\usepackage{tikz}
\usepackage{amsfonts}
\usepackage{enumitem} 
\usepackage{mathrsfs} 
\usepackage[colorlinks=true,urlcolor=blue,
citecolor=red,linkcolor=blue,linktocpage,pdfpagelabels,bookmarksnumbered,bookmarksopen]{hyperref}

\usepackage{collectbox}



\def\e{\varepsilon}
\def\Om{\Omega}

\def \p{\partial}

\def \0{\mathbf{0}}

\newcommand{\R}{{\mathbb R}}

\newcommand{\N}{\mathbb{N}}

\usepackage{subfig}

\newcommand{\dive}{\nabla\cdot}

\newcommand{\be}{\begin{equation}}
\newcommand{\ee}{\end{equation}}

\newcommand{\ba}{\begin{array}}
\newcommand{\ea}{\end{array}}

\usepackage{subfiles}
\newtheorem{remark}{\textbf{Remark}}[section]
\newtheorem{theorem}{\textbf{Theorem}}[section]
\newtheorem{lemma}[theorem]{\textbf{Lemma}}
\newtheorem{corollary}[theorem]{\textbf{Corollary}}
\newtheorem{proposition}[theorem]{\textbf{Proposition}}

\newtheorem{definition}[remark]{\textbf{Definition}}

\newtheorem{innercustomgeneric}{\customgenericname}
\providecommand{\customgenericname}{}
\newcommand{\newcustomtheorem}[2]{%
  \newenvironment{#1}[1]
  {%
   \renewcommand\customgenericname{#2}%
   \renewcommand\theinnercustomgeneric{##1}%
   \innercustomgeneric
  }
  {\endinnercustomgeneric}
}

\newcustomtheorem{customthm}{Framework}

\newtheorem{assumption}{\textbf{Assumption}}
\numberwithin{equation}{section}

\begin{document}
\title[]{A Hele-Shaw Limit With A Variable Upper Bound and Drift}  
\author{Raymond Chu}
\footnote{Department of Mathematics, University of California, Los Angeles}
\date{\today}\email{rchu@math.ucla.edu}
\maketitle


\maketitle

\begin{abstract}
	We investigate a generalized Hele-Shaw equation with a source and drift terms where the density is constrained by an upper-bound density constraint that varies in space and time. By using a generalized porous medium equation approximation, we are able to construct a weak solution to the generalized Hele-Shaw equations under mild assumptions. Then we establish uniqueness of weak solutions to the generalized Hele-Shaw equations. Our next main result is a pointwise characterization of the density variable in the generalized Hele-Shaw equations when the system is in the congestion case. To obtain such a characterization for the congestion case, we derive a new uniform lower bounds on the time derivative pressure of the generalized porous medium equation via a refined Aronson-B{\'e}nilan estimate that implies monotonicity properties of the density and pressure.
\end{abstract}

\section{Introduction}
We consider solutions $\varrho_k$ for $k>1$ of the modified Porous Medium Equation (mPME)
\begin{equation}  \p_t \varrho_k = \nabla \cdot( \varrho_k(\nabla p_k + \vec{b})) + f \varrho_k \text{ in } Q_T := \R^d \times (0,T),  \label{PME_Density} \end{equation} where the pressure $p_k:=\frac{k}{k-1} \left( \frac{\varrho_k}{m} \right)^{k-1}$. In this equation $m, \vec{b}$, and $f$ are given smooth functions (see Section \ref{notation} for the exact assumptions on these functions), while the unknown is the density $\varrho_k$. Here $k>1$ is a constant that reduces congestion by having strong diffusion when the density is larger than the constraint $m$. We can also rewrite this equation as
\begin{equation}  \p_t \varrho_k = \nabla \cdot( m  \nabla v_k^k + \varrho_k \vec{b}) + f \varrho_k \text{ in } Q_T \label{PME_Density_1} \end{equation}
because $\varrho_k \nabla p_k = m \nabla v_k^k$ where $v_k := \varrho_k/m$ is the normalized density. 
For the special case of $m$ being a constant, this equation reduces to the Porous Medium Equation (PME) with a given drift $-\vec{b}$ and source term $f$. This anti-congestion effect has led to the PME being used to model many physical phenomena such as population dynamics \cite{bertsch1984interacting, bertsch1986density, witelski1997segregation} and density based tumor growth \cite{perthame2014hele, david2021incompressible}.  Due to the physical importance of the pressure, we obtain from the constitutive law $p_k = \frac{k}{k-1} \left( \frac{\varrho_k}{m} \right)^{k-1}$ that
\begin{equation} \p_t p_k = |\nabla p_k|^2 + \nabla p_k \cdot \vec{b} + (k-1)p_k \left[ \frac{1}{m}\nabla \cdot(m \nabla p_k) + \frac{\nabla \cdot(m \vec{b}) + mf - \p_t m}{m} \right] \label{PME_pressure1}.
\end{equation}


We are interested in the $\mathit{Hele}$-$\mathit{Shaw}$ limit of $\eqref{PME_Density}$ and $\eqref{PME_pressure1}$, that is the limit as $k \rightarrow \infty$. Formally if we let $(\varrho_{\infty},p_{\infty})$ denote the limit of $(\varrho_k,p_k)$, then we have the following free boundary problem
\begin{equation}  \begin{cases}	\p_t \varrho_{\infty} = \nabla \cdot(m \nabla p_{\infty} + \varrho_{\infty} \vec{b}) + f \varrho_{\infty} \text{ in } Q_T \\ \nabla \cdot(m \nabla p_{\infty}) + \nabla \cdot(m \vec{b}) + mf - \partial_t m = 0 \text{ on } \{p_{\infty}>0\} \\
\varrho_{\infty}(x,t) \leq m(x,t) \text{ in } Q_T \\ p_{\infty}(m-\varrho_{\infty})=0  \text{ in } Q_T  \\   \end{cases} \label{Hele-Shaw}. \end{equation}

The case of $m$ being a constant is well studied and in this case \eqref{Hele-Shaw} is known as the Hele-Shaw equations. The Hele-Shaw equations are  used to model physical phenomena with a density constraint such as congested crowd motion \cite{maury2010macroscopic} and 
 tumor growth  \cite{perthame2014hele, david2021incompressible}. In the  tumor growth models, $m$ represents the maximum packaging density of cells \cite{tang2014composite}, while for the congested crowd motion models, $m$ represents the maximum pedestrian population. By allowing $m$ to vary in space and time, instead of being a constant, it allows these models to account for spatial and time changes to the maximum population or packaging density of cells. \\

\begin{figure}
    \centering
    \includegraphics[width=0.75\textwidth]{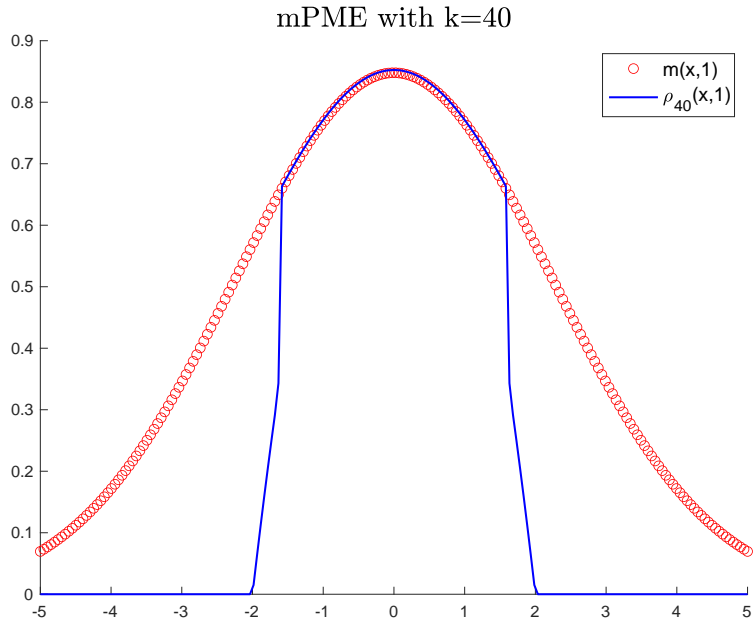}
    \caption{Numerical Simulations of the mPME $\eqref{PME_Density}$ with $f=\vec{b}=0$. In this figure, we have $k=40$ and $m(x,t)=\exp(-t/6-x^2/10)$, which corresponds to the case of congestion, and we see the solution $\varrho_{40}$ becomes saturated on $m$. }
\end{figure}



By combining the techniques in the papers \cite{perthame2014hele, guillen2022hele, david2021incompressible}, we were able to take the Hele-Shaw limit in $L^1(Q_T)$ of $\eqref{PME_Density}$ and $\eqref{PME_pressure1}$ under mild assumption of $m=m(x,t)$,  $\vec{b}$, $f$, and the initial data. In addition, by using  methods from \cite{perthame2014hele}, we were able to derive that the limiting density and pressure are the unique weak solutions to $(\ref{Hele-Shaw})$. Then by using the arguments in \cite{guillen2022hele}, we derive the normal velocity law for $\p \Sigma(t)$ where $\Sigma(t) = \{ x : p_{\infty}(x,t)>0 \}$ in the viscosity sense. The velocity law $V$ is given by
\begin{equation} V =  \left[ -\frac{m \nabla p_{\infty}}{ (m- \varrho^E)} - \vec{b} \right] \cdot \vec{\nu}, \label{velocity_law} \end{equation} where $\vec{\nu}$ is the outward normal of $\p \Sigma(t)$ and $\varrho^E$ is the external density (the limit of $\varrho_{\infty}$ from outside $\Sigma(t)$). The velocity law $(\ref{velocity_law})$ shows that a typical feature of the $\p \Sigma(t)$ is that it can have infinite normal velocity. \\

Then we focus on the case of when
\begin{equation} m_t < mf + \nabla \cdot(m \vec{b}) \text{ in } Q_T, \label{congestion}  \end{equation} which we call the congestion case. This assumption has been used when $m$ is a constant in \cite{kim2019singular} and corresponds to when $m$ grows strictly slower than the external density.  In the congestion case, we were able to  derive an explicit pointwise formula for the limiting density up to $\p \Sigma(t)$ that is consistent with the identification in \cite{kim2019singular} when $m$ is constant. This identification formula implies that when we are the congestion case, if the initial density is a patch, then it is as a patch for all time (that is if $\varrho_{\infty}(x,0)= m(x,0) \chi_{A(0)}$, then $\varrho_{\infty}(x,t) = m(x,t) \chi_{A(t)}$). \\

We were able to explicitly identify the limiting density by deriving a novel uniform Aronson-B{\'e}nilan estimate in the congestion case. This estimate appears to be new even for the case of $m$ being constant with non-zero $\vec{b}$. An important consquence of the estimate is that it leads to a uniform time derivative estimate of the pressure $p_k$ along the streamlines for all $k>1$ is large
\begin{equation} \frac{d}{dt} p_k(X(t,x_0),t) \geq -\left( C + \frac{1}{t} \right) p_k(X(t,x_0),t), \label{new_time_derivative} \end{equation} where the constant $C>0$ is independent of $k$. For each $x_0 \in \R^d$, the streamline $X(t,x_0)$ is defined as the (unique) solution of
\begin{equation}
\begin{cases}
    \frac{d}{dt} X(t,x_0) = - \vec{b}(X(t,x_0),t) \text{ in } (0,\infty) \\ 
    X(0,x_0) = x_0
\end{cases}.
\end{equation} Note in the special case of $\vec{b} \equiv 0$ that the streamlines reduce to the identity map. \\

We use \eqref{new_time_derivative} to show that the normalized density $v_{\infty} := \varrho_{\infty}/m$ is non-decreasing with respect to the streamlines. This means that if $t_1 < t_2$, then $v_{\infty}(X(t_1,x_0),t_1) \leq v_{\infty}(X(t_2,x_0),t_2)$. This estimate also implies the geometric property that $\Sigma(t)$ is non-decreasing with respect to the streamlines. That is if $x_0 \in \Sigma(t)$, then the streamline containing the point $(x_0,t)$ is in $\Sigma(s)$ for all $s \geq t$. The last monotonicity property is enough to allow us to characterize the limiting density up to $\p \Sigma(t)$ pointwise.  \\


Similar uniform time derivative estimates of the pressure are only known in the literature when $m$ is constant in the congestion case with $\vec{b}=0$ or $m$ is constant with $f=\vec{b}=0$. Uniform lower bounds on $\p_t p_k$ similar to $\eqref{new_time_derivative}$ played an essential role in taking the Hele-Shaw limit for weak solutions with no drift in the congestion case \cite{caffarelli1987asymptotic, benilan1989limit, gil2001convergence, gil2003boundary, perthame2014hele, kim2018porous}.  \\

The congestion assumption \eqref{congestion} is essentially necessary to have estimate such as $(\ref{new_time_derivative})$. Indeed, the the congestion assumption is formally equivalent to having $p_{\infty}$ satisfy $\nabla \cdot(m \nabla p_{\infty}) < 0$ inside $\Sigma(t)$ for all $t \geq 0$. If instead there was some time $t_0$ such that $\nabla \cdot(m \nabla p_{\infty})(\cdot,t_0)  \geq 0$ in $\Sigma(t_0)$ where $\Sigma(t_0)$ open, connected, and smooth, then the maximum principle for uniformly elliptic operators implies that $p_{\infty} \equiv 0$. That is the set $\Sigma(t)$ instantaneously collapsed at time $t_0$. However, if the time derivative estimate $(\ref{new_time_derivative})$ was true for this case, then the estimate implies for positive time that $\Sigma(t)$ is non-decreasing with respect to the streamlines, which is a contradiction. So uniform time derivative estimate such as $\eqref{new_time_derivative}$ can not hold in general when there is a $t_0>0$ such that $ [\partial_t m - mf - \nabla \cdot(m \vec{b})](\cdot,t_0) = \nabla \cdot (m \nabla p_{\infty}))(\cdot,t_0) \geq 0$ on $\Sigma(t_0)$. \\

We also mention that it is necessary to use some form of weak theory (viscosity solutions, gradient flow in Wasserstein space, or weak solutions) to understand $(\ref{Hele-Shaw})$. It is well known even for $m$ being a constant that if we start with smooth data, the pressure can become discontinuous in finite time due to topological changes (see \cite{kim2003uniqueness, perthame2014hele} for examples of how the pressure can become discontinuous in finite time). In the recent paper \cite{kim2022regularity}, the authors produced some regularity results of $\p \Sigma(t)$ for constant $m$ under the congestion assumption for continuous viscosity solutions. \\

Now we review the literature on similar Hele-Shaw limits. These Hele-Shaw limits are well studied when $m$ is a constant and until recently, these limits were only investigated when the system is under congestion with $m$ constant or when $f=\vec{b}=0$ with $m(x,t)$ either being constant or $1/t$. The weak limit of $\eqref{PME_Density}$ when  $m \equiv 1$, $f=0$, and $\vec{b}=0$ was studied in \cite{caffarelli1987asymptotic} and was again studied in \cite{benilan1989limit} with more general initial data.  This investigation was continued in \cite{gil2001convergence} and \cite{gil2003boundary} when the equations $(\ref{PME_Density})$ were posed on a domain $\Om \times (0,\infty)$ with general initial data and pressure boundary data.  In \cite{gil2003boundary} the authors considered the case of $m(t)=1/t$ of $(\ref{Hele-Shaw})$ in their analysis. Then in \cite{kim2003uniqueness}, the author showed uniqueness and existence of \eqref{Hele-Shaw} when $m$ is constant with $f$ and $\vec{b}=0$ for viscosity solutions. In addition, in \cite{perthame2014hele} and \cite{kim2018porous} the authors dealt with the Hele-Shaw limits with the condition of $f>0$,  $\vec{b}=0$, and $m \equiv 1$ in a weak and viscosity sense respectively.  In \cite{alexander2014quasi} the authors were able to incorporate a non-zero drift when $\vec{b}(x,t) = \nabla \Phi(x)$ such that $\Delta \Phi > 0$ with $f=0$ and $m \equiv 1$ into their analysis by using viscosity solutions and gradient flows in Wasserstein space. Then in \cite{kim2019singular} the authors were able to deal with the case of $f + \nabla \cdot \vec{b} > 0$ by using viscosity solutions. \\

The above literature primarily focused on the case when $m$ is constant with congestion or $f=\vec{b}=0$. These assumptions were used when $\vec{b}=0$ to obtain similar time derivatives as the one in $\eqref{new_time_derivative}$ to obtain the weak limit. In \cite{kim2019singular}, which dealt with $m \equiv 1$ when there is congestion with drift, $\Sigma(t)$ non-decreasing along streamlines played an essential part in their analysis. As mentioned earlier, these estimates and monotonicity property of $\Sigma(t)$ are no longer generally true when there is a general source term and drift.  \\ 

In \cite{guillen2022hele}, the authors were able to remove the monotonicity conditions when there is a source term $f$ with $\vec{b}=0$ in the exterior of a bounded smooth domain with pressure boundary data and $m$ is constant. In addition, the authors of \cite{david2021incompressible} were also able to take the Hele-Shaw limit with $m$ constant without monotonicity conditions on their source and drift terms.  \\


The authors of \cite{woodhouse2018motion} investigated \eqref{PME_Density} and its Hele-Shaw limit using the theory of gradient flow in Wasserstein space (see \cite{otto2001geometry} for the gradient flow structure of the PME)  when $f=0$, $\vec{b}=0$,  $\partial_t m(x,t) < 0$ (congestion case), and $m(x,t)$ is approximately linear. The approximately linear constraint on $m$ was used in \cite{woodhouse2018motion} to ensure a well defined limit from their JKO approximation scheme and was conjectured to be necessary to take the Hele-Shaw limit of \eqref{PME_Density} and \eqref{PME_pressure1}. In comparison, our assumptions are rather mild (see Section \ref{notation}) and we do not need the approximately linear condition of $m(x,t)$ used in \cite{woodhouse2018motion}. Our methods appear to be the first to be able to take the Hele-Shaw limit for a wide range of $m$ that varies in space and time.


\section{Notation, Assumptions, and Summary of Results} \label{notation}
 Throughout this paper, we denote the pair $(\varrho_k,p_k)$ as weak solutions of \eqref{PME_Density} and \eqref{PME_pressure1} respectively with initial conditions $\varrho_k^{(0)}$ and $p_k^{(0)}$. 
\begin{definition}[Weak Solutions to mPME] We say that $\varrho_k$ is a weak sub solution to   \eqref{PME_Density} with initial data $\varrho_k^0 \in L^2(\R^d)$ if $\varrho_k \in L^2(Q_T)$ and $v_k := \varrho_k/m$ satisfies $v_k^k \in H^1(Q_T)$ such that for all non-negative test functions $\varphi \in C^{\infty}_c(\R^d \times [0,\infty))$
\begin{equation}
	 \iint_{Q_T} -\varrho_k \p_t \varphi \leq  \int_{\R^d} \varphi(t=0) \varrho_k^0(x) + \iint_{Q_T} -m\nabla \varphi \cdot \nabla v_k^k - \varrho_k(\nabla \varphi \cdot \vec{b}) + f \varphi \varrho_k, \label{weak_sub_soln} \end{equation} where $Q_T := \R^d\times (0,T)$.
\end{definition}

Similarly, we say that $\varrho_k$ is a weak super solution if the inequality in \eqref{weak_sub_soln} is reversed. Then we say that $\varrho_k$ is a weak solution if it is both a weak sub and super solution to \eqref{PME_Density}. Now if we denote $p_k := \frac{k}{k-1} \left( \frac{\varrho_k}{m} \right)^{k-1}$, then we say that $p_k$ is a weak solution (respectively super/sub) solution of \eqref{PME_pressure1} if  $\varrho_k$ is a weak (respectively super/sub) solution of \eqref{PME_Density}.
\\ 

We refer the reader to Section \ref{weak_mPME} where we discuss properties of weak solutions of \eqref{PME_Density} such as existence, comparison principle, and smooth approximation. As standard in the PME literature \cite{vazquez2007porous}, when we obtain a priori estimates, we treat $\varrho_k$ as if it is smooth because we can instead use our smooth approximations and take limits to obtain the estimate for our non-smooth $\varrho_k$. \\

 To simplify our notation we let $F := (\nabla \cdot(m \vec{b}) + mf - \p_t m)/m$, then \eqref{PME_pressure1} simplifies to
\begin{equation}
    \p_t p_k = |\nabla p_k|^2 + \nabla p_k \cdot \vec{b} + (k-1)p_k \left( \frac{1}{m} \nabla \cdot(m \nabla p_k) + F \right). \label{PME_pressure}
\end{equation}
\noindent We also let $w_k := \frac{1}{m} \nabla \cdot(m \nabla p_k)$ because this quantity frequently shows up in our calculation. In addition, for many computations, it will be easier to work with the normalized density $v_k := \varrho_k/m$ compared to the density $\varrho_k$. The normalized density solves
\begin{equation} \partial_t v_k = \frac{1}{m}\nabla \cdot(m \nabla v_k^k)+ \nabla v_k \cdot \vec{b} + Fv_k   \text{ in } Q_T. \label{normalized-density} \end{equation}  

We  will refer to the first equation in \eqref{Hele-Shaw} as the density equation and the second equation as the pressure equation. And we will use the letter $C$ to denote a constant that is independent of $k$ that may change from line to line. And we denote $\vec{b} = (b_1,...,b_d)^T$ as the negative drift, $f$ as the source term, and $m(x,t)$ as the hard constraint. In addition, because many computations will involve derivatives of $m$ being divided by $m$, we will use the notation $\lambda := \log(m)$ to simplify our expressions. Also we will also denote for $\tau,t>0$ the sets $Q_t := \R^d \times (0,t)$ and $Q_{\tau,t} := \R^d \times (\tau,t)$ and we fix a $T>0$. Finally the operators $\nabla, \dive$, and $\Delta$ will only be taken in the space variables.   \\

Now we list some assumptions that will be used throughout this paper.

\begin{assumption} [Hard Constraint, Drift, and Source Term Assumptions] We assume that $m(x,t)$, $\vec{b}(x,t)$, and $f(x,t)$ are in $C^{\infty}(Q_T)$ with bounded derivatives of all order such that there is a $\delta > 0$ with $m(x,t) \geq \delta.$ \label{Assum_1}
\end{assumption}

\begin{remark} If $m$ was allowed to go down to zero, the operator $\nabla \cdot(m \nabla p_{\infty})$ in $\eqref{Hele-Shaw}$ would become degenerate elliptic on the zero level set of $m$. But if a Hele-Shaw limit exists when $m$ is allowed to go down to zero, it seems that the zero level set of $m$ will lead to an additional obstacle problem for $\varrho_{\infty}$ on $\p \{x:m(x,t)>0\}$.
\end{remark}

\begin{assumption} [Initial Density and Pressure] We will always assume that the initial density $\varrho_k^{(0)}(x) = \varrho_k(x,0) \in W^{1,1}(\R^d)$ is such that $\varrho_k^0 \geq 0$ almost everywhere. Also we will assume that there is a compact set $K \subset \R^d$ such that for all $k>1$ we have
\[ \text{supp}(\varrho_k^{(0)}) \subset K \] and
\[ \sup_{k>1}||\varrho_k^{(0)}||_{L^{\infty}(\R^d)} \leq C \text{ , } \sup_{k>1}||\varrho_k^{(0)}||_{W^{1,1}(\R^d)} \leq C \text{ , and } \sup_{k>1}||p_k^{(0)}(x)||_{L^{\infty}(\R^d)} \leq C. \]  Finally we assume there is a limiting density $\varrho_{\infty}^{(0)}(x)$ such that $\varrho_k^{(0)} \rightarrow \varrho_{\infty}^{(0)}$ in $L^1(\R^d)$. \label{Assum_2}
\end{assumption}

\begin{remark} The purpose of the condition of $\varrho_k^{(0)} \rightarrow \varrho_{\infty}^{(0)}$ in $L^1(\R^d)$ in Assumption \ref{Assum_2} is to specify the initial data for the density in \eqref{Hele-Shaw}.
\end{remark}
\begin{assumption} [Hard Constraint Restrictions] We assume that either $m(x,t)=m(|x|,t)$ ($m$ is radial in space), $d=1$ (one space dimension), or there is an $R>0$ such that for $\e=\e(\delta)>0$ sufficiently small, we have that $|\nabla m(x,t)| \leq \e/|x|$ for $|x| \geq R$ . $\label{hard_constraint_restriciton}$ 
\end{assumption}

\begin{remark} The only purpose of Assumption $\ref{hard_constraint_restriciton}$ is to construct a super solution of the pressure equation $(\ref{PME_pressure})$ to derive uniform $L^{\infty}$ and compact support bounds on the pressure. 
\end{remark}

With these assumptions, we have the following compactness theorem:

\begin{theorem} If Assumptions $\ref{Assum_1}$, $\ref{Assum_2}$, and $\ref{hard_constraint_restriciton}$ are met, then the sequence of pairs $\{ (\varrho_k,p_k) \}_{k>1}$ has a unique limit in $L^1(Q_T)$ as $k \rightarrow \infty$. Also the limiting densities and pressures are in the regularity class $(\varrho_{\infty},p_{\infty}) \in L^{\infty}( 0,T; L^1(\R^d) \cap L^{\infty}(\R^d)) \times L^2( 0,T; H^1(\R^d) \cap L^{\infty}(\R^d)) $ with compact support. Also the limiting density and pressure $(\varrho_{\infty},p_{\infty})$ are a weak solution to the density equation in $(\ref{Hele-Shaw})$ with initial data $\varrho_{\infty}^{(0)}$. \label{Compactness-Theorem}
\end{theorem} 

\begin{definition} [Weak Solutions of Hele-Shaw Density] We say that the pair $(\varrho_{\infty},p_{\infty})$ is a solution to the density equation with initial data $\varrho_{\infty}^{(0)}$ for $\eqref{Hele-Shaw}$ if for all test functions $\varphi \in C^{\infty}_c(Q_T)$
\[ \int_{\R^d} \varphi(t=0) \varrho_{\infty}^{(0)} + \iint_{Q_T} -(\p_t \varphi) \varrho_{\infty} = \iint_{Q_T} - \nabla \varphi \cdot( m\nabla p_{\infty} + \varrho_{\infty} \vec{b}) + \varphi f \varrho_{\infty}  \] and $\varrho_{\infty}(x,t) \leq m(x,t)$ with $p_{\infty}(m - \varrho_{\infty})=0$ almost everywhere.
\end{definition}

\begin{remark}  For our analysis, we need an extra assumption to be able to show $p_{\infty}$ is a weak solution to the pressure equation in $(\ref{Hele-Shaw})$. 
\end{remark}

We will refer to showing that $p_{\infty}$ is a weak solution to the pressure equation in $(\ref{Hele-Shaw})$ as the complimentary relationship.

\begin{assumption} [Complimentary Relationship Condition] Assume there is a constant $C>0$ independent of $k$ such that
\[ \int_{\R^d} |\nabla p_k(x,0)|^2dx \leq C. \]	\label{complimentary_assumpution}
\end{assumption}

 As in \cite{perthame2014hele} the complimentary relationship is equivalent to $\nabla p_k \rightarrow \nabla p_{\infty}$ in $L^2(Q_T)$. So this control of the $
||\nabla p_k(x,0)||_{L^2(\R^d)}$ seems rather natural. However, in \cite{guillen2022hele} the authors obtained the complimentary relationship without this assumption with no drift and $m$ constant by obtaining an obstacle problem formulation of the pressure. When drift is added into the system, it is not clear if the obstacle problem formulation in \cite{guillen2022hele} still holds. \\

 Assumption \ref{complimentary_assumpution} was also used in \cite{david2021incompressible} to obtain the complimentary relationship with $m$ constant with source and drift terms. We will modify the proof in \cite{david2021incompressible} to obtain the optimal $L^4(Q_T)$ bounds on $\nabla p_k$ (see \cite{david2021free} for an example of the optimality of the bounds), which will give us enough compactness to obtain the complimentary relationship.

\begin{theorem} [Complimentary Relationship] Assume that Assumptions $\ref{Assum_1}$, $\ref{Assum_2}$,  $\ref{hard_constraint_restriciton}$, and $\ref{complimentary_assumpution}$ are met, then for any $\varphi \in C^{\infty}_c(Q_T)$ we have that
\[ \iint_{Q_T} -mp_{\infty}[\nabla \varphi \cdot \nabla p_{\infty}] - m \varphi |\nabla p_{\infty}|^2 + \varphi p_{\infty} mF = 0. \] That is the limiting pressure satisfies the complimentary relationship
\[ p_{\infty} \left[ \nabla \cdot(m \nabla p_{\infty}) + mF \right]  = 0 \text{ in } Q_T \] in the sense of distributions.
\label{Complimentary-Theorem}
\end{theorem}

 Then with the existence and uniqueness of the generalized Hele-Shaw problem, we focus on geometric properties of $\Sigma(t) := \{x:p_{\infty}(x,t)>0\}$. 

 \begin{theorem} The normal velocity $V$ of the free-boundary $\p \Sigma(t)$ satisfies
 \[ V =  \left[ - \frac{m \nabla p_{\infty}}{(m-\varrho^E)} - \vec{b} \right] \cdot \vec{\nu} \] in the sense of comparison with classical barriers where $\varrho^E$ is the external density. The external density is the limit of the density $\varrho_{\infty}$ from outside the saturated region $\Sigma(t)$. See Section $\ref{Velocity_Law}$ for a definition of classical barriers. \label{velocity-law-theorem}
 \end{theorem}
 
 Then we focus on the case of congestion \eqref{congestion}.
This constraint is a generalization of the one used in \cite{kim2019singular} used for $m \equiv 1$. And in this case, we can identify the limiting pressure up to $\p \Sigma(t)$ when we have control of the initial limiting pressure:
\begin{theorem} [Explicit Identification of $\varrho_{\infty}$ in The Congestion Case] Assume that Assumptions $\ref{Assum_1}$, $\ref{Assum_2}$, and $\ref{hard_constraint_restriciton}$, are met such that $p_k(x,0) \rightarrow p_{\infty}(x,0)$ uniformly such that $\Sigma(0)$ is open and we are in the congestion case $(\ref{congestion})$. Then we have the pointwise formula
\[ \varrho_{\infty}(x,t) = \begin{cases}  \varrho^E(x,t) \text{ for a.e. } (x,t) \in \text{Int}(\{p(x,t)=0\}) \\ m(x,t) \text{ for a.e. } (x,t) \in \{p(x,t)>0\}  \end{cases}, \] where the external density $\varrho^E(x,t)$ solves the continuity equation
\[ \begin{cases} \p_t \varrho^E = \nabla \cdot( \varrho^E \vec{b}) + f \varrho^E \text{ in } Q_T \\ \varrho^E(x,0) = \varrho_{\infty}(x,0)  \end{cases}. \label{identification} \] \end{theorem} 
 
 \begin{remark} Note that this identification Theorem $\ref{identification}$ is consistent with the one in \cite{kim2019singular}, which dealt with the case of $m \equiv 1$ using viscosity solutions.
 \end{remark}
 
 We were able to obtain this identification by refining the new Aronson-B\'enilan estimate for the PME with drift found in \cite{kim2021porous}. We showed that under the congestion case  for all $k>1$ there exists a $C>0$ independent of $k$ such that
 \[ \frac{1}{m} \nabla \cdot(m \nabla p_k) \geq -F - \frac{C}{k-1} - \frac{1}{(k-1)t}. \] Now our Aronson-B\'enilan estimate implies that we have the following uniform estimate of the pressure $p_k$ along the streamlines $X(t,x_0)$ for all $k>1$ for some $C>0$ independent of $k$
 \[ \frac{d}{dt} p_k(X(t,x_0),t) \geq -(C+\frac{1}{t})p_k(X(t,x_0),t).   \]  In particular, we deduce that $\Sigma(t)$ is non-decreasing with respect to streamlines for positive time. That is if $x \in \Sigma(t)$ for $t>0$ then the streamline containing $x$ is in $\Sigma(s)$ for all $T \geq s \geq t$. Then by using the uniform convergence of the initial limiting pressure and $\Sigma(0)$ is open, we can use the barrier constructed in Lemma \ref{Big_Pi_Barrier} for the congestion case to conclude that $\Sigma(t)$ is non-decreasing with respect to the streamline for all $0 \leq t \leq T$. This is enough for us to conclude the identification of $\varrho_{\infty}$.
  
  \subsection{Open Problems} Let us denote $(\varrho_{\infty},p_{\infty})$ as the limiting density and pressure respectively.  
 \begin{itemize}[label=$\circ$] 
 
 \item $\mathit{Obstacle}$ $\mathit{Problem}$ $\mathit{Formulation}$ $\mathit{of}$  $\mathit{the}$ $\mathit{Pressure}$ In \cite{guillen2022hele} the authors derived an obstacle problem formulation of $p_{\infty}$ when $\vec{b} \equiv 0$ and $m \equiv 1$. However, when there is a non-zero drift $\vec{b}$, the proof used in \cite{guillen2022hele} does not carry over. \\
 
 \item $\mathit{When}$ $\mathit{m}$ $\mathit{is}$ $\mathit{allowed}$ $\mathit{to}$ $\mathit{go}$ $\mathit{to}$ $\mathit{zero}$ When $m$ is allowed to reach zero, the limiting PDE $(\ref{Hele-Shaw})$ seems to be modified into a new obstacle problem that includes the zero level set of $m$.

 \end{itemize}
\subsection{Acknowledgements}
 The author was partially supported by NSF grant DMS-1900804 and the Summer Mentored Research Fellowship at the University of California, Los Angeles during this project.  The author is also very grateful to his advisor Inwon Kim for suggesting this project,  for her guidance and mentorship throughout this project, and for reading the manuscript.  The author is also very grateful to the two anonymous referees for their very detailed  comments.

 \section{Existence of the Hele-Shaw Limit}

\subsection{\texorpdfstring{$L^{\infty}(Q_T)$}{L} Bounds and Finite Speed of Propagation}
We will derive enough estimates to be able to pass the limit in $(\ref{PME_Density})$ and $\eqref{PME_pressure}$ to $(\ref{Hele-Shaw})$ as $k \rightarrow \infty$ in the weak sense. Our first estimate will be used to obtain $L^{\infty}(Q_T)$ and control the support of the pressure $p_k$. \\

First notice as $p_k(x,0) \geq 0$ and $0$ is a solution to the pressure equation, we may apply the comparison principle to deduce that $p_k(x,t) \geq 0$ almost everywhere. Now we will show the following bounds:

\begin{lemma}[$L^{\infty}(Q_T)$ Bounds and Finite Speed of Propagation] $\label{Pressure_Bounds}$ If Assumptions $\ref{Assum_1}$, $\ref{Assum_2}$, and $\ref{hard_constraint_restriciton}$ are met then there is a constant $C=C(T)$ and a compact set $\Om \subset \mathbb{R}^d$ independent of $k>1$ such that
\[ 0 \leq p_k \leq C  \text{ a.e. in } Q_T, \quad \text{supp}(p_k(t)) \subset \Om \text{ for a.e. } 0 \leq t \leq T. \]
\end{lemma}

\begin{proof}
To prove this lemma, we use the viscosity super solution constructed in Lemma $\ref{super_solution_barrier_radial}$ and $\ref{super_solution_case_2}$. We choose the initial conditions of this barrier so that the barrier at time $0$ is larger than $\sup_{k>1} ||p_k^{(0)}||_{L^{\infty}(\R^d)}$ inside $K$. So by the comparison principle, $p_k$ is dominated by the barrier. Then as the barrier is compactly supported and bounded for all times, we deduce the same is true for $p_k$.
\end{proof}

This immediately implies compact support and $L^{\infty}(Q_T)$ bounds on the density and normalized density.


\begin{corollary} The density $\varrho_k$ and normalized density $v_k$ are also uniformly bounded in $L^{\infty}(Q_T)$, uniformly compactly supported in $Q_T$, and non-negative almost everywhere. \label{density_bounds}
\end{corollary}

\subsection{BV Estimates on the Pressure and Density} Now we will derive $L^2(Q_T)$ estimates on $|\nabla p_k|$, which follows from just integrating \eqref{PME_pressure}. Then we obtain $L^1(\R^d)$ estimates for $|\nabla v_k(t)|$ for a.e. $t \in [0,T]$ via an integral Gr\"onwall's inequality with Kato's inequality argument.

\begin{lemma} [$L^2(Q_T)$ bounds for $|\nabla p_k|$] \label{L2_Pressure_Gradient} If Assumptions $\ref{Assum_1}$, $\ref{Assum_2}$, and $\ref{hard_constraint_restriciton}$ are met then for all $k \geq 6$  there is a constant $C=C(T)$ such that
\[ \iint_{Q_T} |\nabla p_k|^2   \leq C(T). \]
\end{lemma}


\begin{proof} First we use $\frac{1}{m} \nabla \cdot(m \nabla p_k) = \Delta p_k + \nabla \lambda \cdot \nabla p_k$ and integrate the pressure equation over $Q_T$ to see that
\[ (k-2) \iint_{Q_T}  |\nabla p_k|^2 = \int_{\R^d} p_k(x,0)-p_k(x,T) + \iint_{Q_T} \nabla p_k \cdot (\vec{b} + (k-1) p_k \nabla \lambda) + (k-1)p_k F . \] Now we recall that Young's Inequality with $\e>0$ says that $|\vec{v}_1 \cdot \vec{v}_2| \leq \e |\vec{v}_1|^2  + |\vec{v}_2|^2/(4\e)$. Using Young's Inequality with $\e$  combined with our uniform sup bounds and uniform support control on $p_k$ (Lemma $\ref{Pressure_Bounds}$) gives us that there is $C=C(T)$ such that
\[ (k-2) \iint_{Q_T} |\nabla p_k|^2 \leq kC(T) + \frac{k+1}{2}\iint_{Q_T} |\nabla p_k|^2.  \] Now because $k \geq 6$, we have that $(k-2)-\frac{k+1}{2} \geq  1/2$ and $\sup_{k \geq 6}$ $((k-2)-\frac{k+1}{2})/k \leq \tilde{C}$, which implies the desired bound.

\end{proof}

Note that this $L^2(Q_T)$ estimate combined with H\"older's inequality and our uniform compact support of $p_k$ (see Lemma \ref{Pressure_Bounds}) gives us $L^1(Q_T)$ control of $\nabla p_k$. Thus to obtain $L^1(Q_T)$ convergence of the pressure along a sub-sequence as $k \rightarrow \infty$, it remains to show time derivative bounds on the pressure. Uniform time derivative bounds on $p_k$ is one of the most delicate parts of our analysis. This is unsurprising because in \cite{perthame2014hele} the authors constructed examples of $p_{\infty}$ that have strong jump discontinuities in time. \\

Our first step in obtaining time derivative control of $p_k$ is to obtain BV bounds for the normalized density $v_k$. To derive our BV bounds, we use the following  lemma that shows $|\nabla v_k^k| \approx |\nabla p_k|$.

\begin{lemma} If Assumptions $\ref{Assum_1}$, $\ref{Assum_2}$, and $\ref{hard_constraint_restriciton}$ are met with $k \geq 6$ then there is a constant $C=C(T)$ such that for the normalized density $v_k$ we have that
\[ \iint_{Q_T} |\nabla v_k^k|^2 \leq C(T). \] \label{vk_L2}
\end{lemma}

\begin{proof} Note that from Corollary $\ref{density_bounds}$ we have that $0 \leq v_k \leq C$ almost everywhere, so 
\[ |\nabla v_k^k| \leq Ckv_k^{k-2} |\nabla v_k| = C|\nabla p_k|.  \] Hence, by integrating this inequality and using Lemma $\ref{L2_Pressure_Gradient}$, we conclude the desired inequality.
\end{proof} 

Now we have enough estimates to obtain $L^1(\R^d)$ bounds on $|\nabla v_k(t)|$ for a.e. $t \in [0,T]$.

\begin{lemma} [BV bounds for the Normalized Density] Assume that Assumptions $\ref{Assum_1}$, $\ref{Assum_2}$, and $\ref{hard_constraint_restriciton}$ are met with $k \geq 6$. Then there is a constant $C=C(T)$ such that for almost every $t \in [0,T]$
\[ \int_{\R^d} |\nabla v_k(t)|  \leq C(T). \] \label{density_bv}
\end{lemma}

\begin{proof} We differentiate \eqref{normalized-density} in $x_i$ using that $\frac{1}{m} \nabla \cdot(m \nabla v_k^k) = \Delta(v_k^k) + \nabla \lambda \cdot \nabla v_k^k$ to  obtain for $\omega_i := \p_{x_i} v_k$ that
\begin{equation} \p_t \omega_i =  \Delta(kv_k^{k-1} \omega_i) + \nabla (\p_{x_i} \lambda) \cdot \nabla v_k^k + \nabla \lambda \cdot \nabla(kv_k^{k-1} \omega_i) + \nabla(\omega_i) \cdot \vec{b} + \nabla v_k \cdot  (\p_{x_i} \vec{b}) + v_k[\p_{x_i} F] + \omega_i F. \label{omega_i_eqn} \end{equation} Note that our assumptions tells us that $F$ is a smooth bounded function with bounded derivatives. To deal with the term $\nabla \lambda \cdot \nabla(kv_k^{k-1} \omega_i)$, we observe that
\[ \nabla \lambda \cdot \nabla (kv_k^{k-1} \omega_i) = \nabla \cdot( kv_k^{k-1} \omega_i \nabla \lambda) - kv_k^{k-1} \omega_i \Delta \lambda =  \nabla \cdot( kv_k^{k-1} \omega_i \nabla \lambda) - (\p_{x_i} v_k^k) \Delta \lambda. \] Now we multiply \eqref{omega_i_eqn} by $\text{sign}(\omega_i)$ and use Kato's Inequality, Cauchy-Schwarz, and $v_k \leq C$ (see Corollary \ref{density_bounds}) to see that\[ \p_t |\omega_i| \leq \Delta(kv_k^{k-1} |\omega_i|) + C|\nabla v_k^k| + \nabla \cdot(k v_k^{k-1} |\omega_i| \nabla \lambda) + C |\p_{x_i} v_k^k|  + \nabla |\omega_i| \cdot \vec{b} + C| \nabla v_k| + C|\omega_i|  + C. \]
Hence, integrating this using that $\omega_i$ is compactly supported and that $|\nabla v_k^k| \in L^1(Q_t)$ uniformly in $k$ (see Lemma \ref{vk_L2}) gives
\[ \frac{d}{dt} \iint_{Q_t} |\omega_i|  \leq C + \int_{\R^d} |\omega_i|(t=0) +  \iint_{Q_t} -|\omega_i| (\nabla \cdot \vec{b}) + C|\omega_i| + C \left( \sum_{j=1}^{d} |\omega_j|^2 \right)^{1/2}. \] 
Summing this over all $1 \leq i \leq d$ with $\omega := \sum_{j=1}^{d} |\omega_j|$ along with our initial BV bounds (Assumption \ref{Assum_2}) gives
\[ \frac{d}{dt} \iint_{Q_t} \omega  \leq C + C \iint_{Q_t} \omega.  \] Now Gr\"{o}nwall's Inequality implies with our initial BV bounds (Assumption $\ref{Assum_2}$)  the desired bounds.
\end{proof}


\subsection{Generalized Semi-Convexity Estimate} \label{general-semi-convexity} Now we aim to control the time derivatives of the normalized density and pressure. To do this, we need a semi-convexity estimate in a similar vein to the Aronson-Bénilan estimate. The famous Aronson-Bénilan (see \cite{vazquez2007porous}) estimate says that for the case of $m(x,t) \equiv 1, f = 0$, and $\vec{b} = \vec{0}$ one has the following inequality
\[ \Delta p_k \geq -\frac{1}{(k-1)t}. \] We will derive for our case the following similar estimate
\[ w_k = \frac{1}{m} \nabla \cdot(m \nabla p_k) \geq -\frac{2}{(k-1)t} - K_1, \] where $K_1>0$ is a constant independent of $k$. This generalized semi-convexity estimate will be used to show for any $\tau>0$ and a.e. $t \in [\tau,T]$ that $||\p_t v_k(t)||_{L^1(\R^d)} \leq C(\tau,T)$. For this section, we fix a $k \geq 6$ and will drop the subscript $k$ on the $p_k$ and $w_k$. \\

The main idea of the generalized semi-convexity estimate is to derive that $w$ solves a parabolic PDE and modify the barrier introduced in \cite{kim2021porous}. In Section \ref{Characterization}, we modify the barrier used in this section under the congestion assumption to derive a refined Aronson-Bénilan estimate. \\

We begin deriving the PDE that $w$ solves. Observe that
\begin{equation} w_t = \nabla \lambda_t \cdot \nabla p + \nabla \lambda \cdot \nabla p_t + \Delta p_t. \label{basic_w_eqn} \end{equation}  Now we will derive expressions for $\nabla p_t$ and $\Delta p_t$. By differentiating \eqref{PME_pressure} in $x_i$ we see that
\begin{equation} \p_{x_i} p_t = \sum_{j=1}^{d} \left[ 2(\p_{x_j} p)(\p^2_{x_j x_i} p) + (\p^2_{x_i x_j} p) b_j + (\p_{x_j} p)(\p_{x_i} b_j) \right] + (k-1) \left[ (\p_{x_i} p)(w+F) +p(\p_{x_i} w + \p_{x_i} F) \right]. \label{nabla_pt} \end{equation} Differentiating in $x_i$ again and summing gives
\begin{equation} \Delta p_t = \sum_{j=1}^{d} 2(\p^2_{x_i x_j} p)^2 + 2\nabla p \cdot \nabla(\Delta p) + \nabla(\Delta p) \cdot \vec{b} + \nabla p \cdot(\Delta \vec{b}) + 2\sum_{i,j=1}^{d} (\p^2_{x_i x_j} p)(\p_{x_i} b_j) +  \label{delta_pt1}  \end{equation}
\begin{equation} + (k-1) \left\{ \Delta p(w+F) + 2 \nabla p \cdot (\nabla w + \nabla F) + p(\Delta w + \Delta F) \right\}.  \label{delta_pt2}\end{equation} 
\noindent Using \eqref{nabla_pt}, \eqref{delta_pt1}, and \eqref{delta_pt2} in \eqref{basic_w_eqn}, we obtain the following PDE for $w$
    \begin{equation} w_t = \nabla \lambda_t \cdot \nabla p + (k-1) \left[ w(w+F) + (\nabla w + \nabla F) \cdot(p \nabla \lambda + 2 \nabla p) + p(\Delta w + \Delta F) \right] + \label{w1} \end{equation}
    \begin{equation}
        + \sum_{i,j=1}^{d} 2(\p^2_{x_i x_j} p)^2 + 2(\nabla w \cdot \nabla p - (\nabla p)^T D^2 \lambda (\nabla p)) + \nabla(w-\nabla p \cdot \nabla \lambda) \cdot \vec{b} + \nabla p \cdot(\Delta \vec{b})  +  \label{w2}    \end{equation} 
    \begin{equation}
        + \sum_{i,j=1}^{d} (\p^2_{x_i x_j} p)(b_j)(\p_{x_i} \lambda) + \sum_{i,j=1}^{d} (\p_{x_j} p)(\p_{x_i} b_j)(\p_{x_i} \lambda) + 2 \sum_{i,j=1}^{d} (\p^2_{x_i x_j} p)(\p_{x_i} b_j). \label{w3} \end{equation} To derive this equation we used the identity
\[ \sum_{i,j=1}^{d} 2(\p_{x_j} p)(\p^2_{x_i x_j}p)(\p_{x_i} \lambda) + 2\nabla p \cdot \nabla (\Delta p) = 2\{ \nabla w \cdot \nabla p - (\nabla p)^T D^2\lambda (\nabla p)  \}, \]  which follows from expanding the right hand side. \\

Now we will use \eqref{w1},\eqref{w2}, and \eqref{w3} to derive that $w$ is a super solution to a simpler parabolic differential equation.  Indeed, by Cauchy-Schwarz and Young's Inequality with $\e$ we have that
\begin{equation} \nabla \lambda_t \cdot \nabla p \geq -\frac{1}{2 \e} ||\nabla \lambda_t||^2_{\infty} - \frac{\e}{2} |\nabla p|^2.  \end{equation} In addition, we have the following bound
\begin{equation}\nabla(\nabla p \cdot \nabla \lambda) \cdot \vec{b} = \sum_{i,j=1}^{d} (\p^2_{x_i x_j} p)(\p_{x_i} \lambda)b_j + (\p^2_{x_i x_j} \lambda)(\p_{x_i} p) b_j   \end{equation}
\begin{equation} \geq \sum_{i,j=1}^{d} - \frac{\e}{2} |\p^2_{x_i x_j}  p|^2 - \frac{1}{2 \e} \left( ||\p_{x_i} \lambda||^2_{\infty} ||b_j||^2_{\infty} \right) - \frac{\e}{2} |\nabla p|^2 - \frac{1}{2 \e} \sum_{i,j=1}^{d} ||\p^2_{x_i x_j} \lambda||^2_{\infty} ||b_j||^2_{\infty}.  \end{equation} Also we observe that
\begin{equation}
\nabla p \cdot(\Delta \vec{b}) \geq -\frac{\e}{2} |\nabla p|^2- \frac{1}{2 \e} ||\Delta \vec{b}||^2_{\infty}. \end{equation} Then by using Young's Inequality with $\e$, we see that
\begin{equation} \sum_{i,j=1}^{d} (\p_{x_j} p)(\p_{x_i} b_j)(\p_{x_i} \lambda) + 2 \sum_{i,j=1}^{d} (\p^2_{x_i x_j} p)(\p_{x_i} b_j) \geq -C(\e) -\e |\nabla p|^2 - \e \sum_{i,j=1}^{d} |\p^2_{x_i x_j}p|^2. \end{equation}  Now by using that $p \geq -C$ (see Lemma \ref{Pressure_Bounds})
\begin{equation}  w(w+F) \geq \frac{w^2}{2} - C, \quad \nabla F \cdot (p \nabla \lambda + 2\nabla p) \geq -C -  \frac{|\nabla p|^2}{2}, \text{ and } p \Delta F \geq -C. \end{equation} Combining these inequalities along with making $\e$ sufficiently small (independent of $k$) gives
\[ \p_t w \geq -|\nabla p|^2 + \nabla w \cdot \vec{b} + (k-1) \left\{ \frac{w^2}{2} + p\nabla w \cdot \nabla \lambda + 2 \nabla w \cdot \nabla p + p \Delta w - \frac{|\nabla p|^2}{2} - C \right\} +  \sum_{i,j=1}^{d} (\p^2_{x_i x_j} p)^2 + \] \[ + 2(\nabla w \cdot \nabla p - (\nabla p)^T D^2 \lambda (\nabla p)). \] Now we treat $p$ as a known function and define the parabolic differential operator
\begin{equation} L[\phi] := \p_t\phi - \nabla w \cdot \vec{b} + |\nabla p|^2 - 2(\nabla \phi \cdot \nabla p - (\nabla p)^T D^2 \lambda(\nabla p)) + \label{Lw1} \end{equation}
\begin{equation} - (k-1) \left\{ \frac{\phi^2}{2} + \nabla \phi \cdot(p \nabla \lambda + 2 \nabla p) + p \Delta \phi - \frac{|\nabla p|^2}{2} - C \right\},  \label{Lw2} \end{equation} then we have by construction that $L[w] \geq 0$. Now we construct a sub-solution of this operator using the barrier introduced in \cite{kim2021porous}.

\begin{lemma}[Subsolution for semi-convexity estimate] We denote $p=p_k$ and $w=w_k$ for ease of notation. Then there exists a sufficiently large $K_1>0$ such that for any $\tau>0$
\[ \psi(x,t;\tau) := -\frac{2}{(k-1)(t+\tau)} + p - K_1 \] satisfies $L[\psi] \leq 0$ whenever $k>1$ is sufficiently large where $L$ is the parabolic differential operator in \eqref{Lw1} and \eqref{Lw2}. \label{subsolution}
\end{lemma}

\begin{remark}  This barrier $\psi(x,t;\tau)$ is inspired by the one in \cite{kim2021porous}.
\end{remark}

\begin{proof} We plug in $\psi$ to $L$ to see that
\[ L[\psi] = \frac{2}{(k-1)(t+\tau)^2} + 2(\nabla p)^T D^2 \lambda (\nabla p) + (k-1)(pF - \frac{\psi^2}{2} - \frac{3}{2} |\nabla p|^2 + C  ). \] Now we use that
\[ \psi^2 = \frac{4}{(k-1)^2(t+\tau)^2} - 2(p-K_1)\left(\frac{2}{(k-1)(t+\tau)} \right) + (p-K_1)^2 \]
\[ \geq  \frac{4}{(k-1)^2(t+\tau)^2} + (p-K_1)^2, \] where we chose $K_1 > \sup_{k > 1} ||p_k||_{L^{\infty}(Q_T)}$ (which is finite thanks to Lemma \ref{Pressure_Bounds}), so that $p-K_1 \leq 0$. Now from $k \geq 2$, we have that $(k-1) \frac{-2}{(k-1)^2(t+\tau)^2} \leq \frac{-2}{(k-1)^2(t+\tau)^2}$, which implies
\[ L[\psi] \leq \underbrace{\frac{1}{(k-1)(t+\tau)^2} \left[2-2\right]}_{0} + \underbrace{2(\nabla p)^T D^2 \lambda (\nabla p) + (k-1)\left[-\frac{3}{2} |\nabla p|^2 \right]}_{\mathcal{I}_1} + \] \[ + \underbrace{(k-1)(-\frac{(p-K_1)^2}{2} + C)}_{\mathcal{I}_2}.  \] Note that we used the  $L^{\infty}(Q_T)$ pressure bounds (Lemma $\ref{Pressure_Bounds}$) to have $pF \leq C$ for some $C>0$. Now for all $(x,t) \in Q_T$, let $\lambda_1(x,t) \geq \lambda_2(x,t) \geq ... \geq \lambda_n(x,t)$ be the eigenvalues of $D^2\lambda(x,t)$. Then let $\Lambda := \sup_{(x,t) \in Q_T} \{ \lambda_1(x,t)  \}$ be the biggest eigenvalue of $D^2\lambda$ over $Q_T$. This implies
\[ \mathcal{I}_1 \leq 2 \Lambda |\nabla p|^2 -(k-1) \frac{3}{2} |\nabla p|^2 \leq 0 \text{ for } k \geq \frac{4}{3} \Lambda + 1. \]  Then for $\mathcal{I}_2$, by making $K_1$ larger if necessary, we can also assume that $K_1 \geq \sup_{k > 1} ||p_k||_{L^{\infty}(Q_T)} + \sqrt{2C}$ , which implies $\mathcal{I}_2 \leq 0$. So we conclude that $L[\psi] \leq 0$.
\end{proof}
Now we reintroduce the subscript $k$.
\begin{lemma} [Semi-Convexity Estimate] There exists a $K_1$ independent of $k$ such that we have
\[  w_k := \frac{1}{m} \nabla \cdot(m \nabla p_k) \geq -\frac{2}{(k-1)t} - K_1 \] for sufficiently large $k$ in the distribution sense. \label{semi-convexity}
\end{lemma}

\begin{proof} Using the $\psi(x,t;\tau)$ from Lemma $\ref{subsolution}$, we see that
\[ \lim_{\tau \downarrow 0} \psi(x,0;\tau) = -\infty. \] This implies for small enough $\tau=\tau(k)>0$ we can compare $\psi(\cdot,\cdot,\tau)$ to $w_k$ (because $L[w_k] \geq 0$ and $L[\psi(\cdot,\cdot;\tau)] \leq 0$) to obtain the claim.
\end{proof}

This gives us enough control to derive an $L^1(Q_{\tau,T})$ estimate on $w_k$.

\begin{lemma} [$L^1(Q_{\tau,T})$ bound on the semi-convexity term]  Assume that Assumptions $\ref{Assum_1}$, $\ref{Assum_2}$, and $\ref{hard_constraint_restriciton}$ are met then if $\tau>0$, there is a constant $C=C(\tau,T)$ such that for all sufficiently large $k$
\[ \iint_{Q_{\tau,T}} |w_k|  \leq C(\tau,T). \] \label{L1_semi_convexity}
\end{lemma}

\begin{proof}
Recall that $|w_k| = w_k + 2|w_k|_{-}$ where $|x|_{-} := -\min(x,0)$ for $x \in \R$. Then our semi-convexity estimate Lemma \ref{semi-convexity} implies that $|w_k|_{-} \leq  \frac{2}{(k-1) \tau} + K_1$ in $Q_{\tau,T}$. Now using that $w_k$ is compactly supported thanks to Lemma \ref{Pressure_Bounds}, we conclude the claim by integrating.
\end{proof}

\subsection{Time Derivative Estimates} Now equipped with our semi-convexity estimate we can derive $L^1(\R^d)$ estimates on $|\p_t v_k(t)|$ for a.e. $t \in [\tau,T]$ and $L^1(Q_{\tau,T})$ estimates on $|\p_t p_k|$. Note that we cannot expect to go down to $t=0$ without stronger conditions on our initial data. \\

To begin we first observe the semi-convexity estimate Lemma $\ref{semi-convexity}$ that
\[ \p_t p_k \geq \vec{b} \cdot \nabla p_k + (k-1)p_k \left[ -\frac{2}{(k-1)t} - K_1\right].   \] Then by using $p_k = \frac{k}{k-1} (v_k^{k-1})$ we have that
\[ \p_t v_k \geq \vec{b} \cdot \nabla v_k - v_k \left[K_1 + \frac{2}{(k-1)t} \right]. \]  Cauchy-Schwarz implies that
\[ \p_t v_k \geq -C|\nabla v_k| - v_k \left[K_1 + \frac{2}{(k-1)t} \right]. \] This estimate implies the following bound on the negative part of $\p_t v_k$,
\[  |\p_t v_k|_{-} \leq C|\nabla v_k| + |v_k| \left[ K_1 + \frac{2}{(k-1) t} \right]. \] Thus our BV bounds on $v_k$ (Lemma \ref{density_bv}) and $L^{\infty}(Q_T)$ bounds on $v_k$ (Corollary $\ref{density_bounds}$), implies for a.e. $t \in [\tau,T]$ that
\[ \int_{\R^d}  |\p_t v_k|_{-}(t) dx \leq C(\tau,T). \] Now we integrate in space \eqref{normalized-density} to see that
\[ \left| \int_{\R^d} \p_t v_k \right|   \leq C.  \] So using that $|\p_t v_k| = \p_t v_k + 2 |\p_t v_k|_{-}$, we see for a.e. $t \in [\tau,T]$ that
\[ \int_{\R^d} |\p_t v_k|(t)  \leq C(\tau,T). \] That is we have the following lemma:

\begin{lemma} [Time Derivative Control of The Normalized Density]  Assume that Assumptions $\ref{Assum_1}$, $\ref{Assum_2}$, and $\ref{hard_constraint_restriciton}$ are met, then for any $\tau>0$ we have the bound for almost every $t \in [\tau,T]$
\[ \int_{\R^d} |\p_t v_k|(t) \leq C(\tau,T)  \] for sufficiently large $k$. $\label{density_time}$

\end{lemma}
Now we will obtain $L^1(Q_{\tau,T})$ estimates of $\p_t p_k$. But we first need the following lemma based on an estimate in \cite{guillen2022hele}.

\begin{lemma}   Assume that Assumptions $\ref{Assum_1}$, $\ref{Assum_2}$, and $\ref{hard_constraint_restriciton}$ are met and let $\tau>0$. Then if we write $q:=\p_t v$ (we drop the subscript $k$ in this lemma). Then we have for sufficiently large $k$ and almost every $0<\tau<T$ that
\[ \iint_{Q_{\tau,T}} kv^{k-1} |q| \leq C(\tau,T). \] $\label{pressure_time_lemmma}$
\end{lemma}

\begin{proof}
 Observe by differentiating the normalized density equation $(\ref{normalized-density})$ in time that we have
\[ \p_t q = \Delta(kv^{k-1} q) + \nabla \lambda_t \cdot \nabla v^k + \nabla \lambda \cdot \nabla (kv^{k-1} q) + \nabla q \cdot \vec{b} + \nabla v \cdot (\p_t \vec{b}) + q F + v (\p_t F).  \] Multiplying by $\text{sign}(q)$ and using Kato's inequality implies that
\[ \p_t |q| \leq \Delta(kv^{k-1}|q|) +|\nabla \lambda_t  \cdot \nabla v^k| + \nabla \lambda \cdot \nabla(kv^{k-1}|q|) + \nabla \cdot(|q| \vec{b}) - |q| \nabla \cdot \vec{b} \]
\[ + |\nabla v \cdot (\p_t \vec{b})| + C|q| + C|v|. \] Now we observe that $\Delta(kv^{k-1}|q|) + \nabla \lambda \cdot \nabla (kv^{k-1}|q|) =  \nabla \cdot(m \nabla (kv^{k-1}|q|))/m$. Hence, we have from Young's Inequality and Cauchy-Schwarz for any smooth $\phi = \phi(x) \geq 0$ that is compactly supported
\[ \phi (\p_t |q|) \leq \frac{\phi}{m} \nabla \cdot(m \nabla (kv^{k-1} |q|)) + \phi( \frac{|\nabla v^k|^2}{2} + C) + \phi\left[ \nabla \cdot(|q| \vec{b}) - |q| \nabla \cdot \vec{b} \right] + C \phi |\nabla v| + C \phi |q| + C \phi |v|. \] So by integration by parts and using our $L^2(Q_T)$ bounds on $|\nabla v^k|$ (Lemma \ref{vk_L2}), BV bounds on $|\nabla v|$ (Lemma \ref{density_bv}),  $L^{\infty}(Q_T)$ bounds on $v$ (Corollary \ref{density_bounds}), and $L^1(Q_{\tau,T})$ bounds on $q$ (Lemma \ref{density_time}) we have that
\[ \iint_{Q_{\tau,T}} \phi(\p_t |q|) \leq \iint_{Q_{\tau,T}} (kv^{k-1}|q|) [ \nabla \cdot(m \nabla(\phi/m))  ]   + C(\phi,\tau,T).  \] We now let $R>0$ be so large such that 
\[ \text{supp}(v_k) \subset B_R(0) \text{ for all } k>1 \text{ for } t \in [0,T]. \]  Then we define $\phi = m\varphi$ where $\varphi$ is the unique solution of the following uniformly elliptic PDE
\[
\begin{cases}
\nabla \cdot(m \nabla \varphi) = -1 \text{ in } B_{2R}(0) \\
\phi = 0 \text{ on } \p B_{2R}(0)
\end{cases}
\] and defined to be 0 outside of this ball.   
Hence, we deduce that
\[ \iint_{Q_{\tau,T}} (kv^{k-1}|q|) \leq \int_{\R^d} \phi(x) \left[ |q|(x,\tau)-|q|(x,T) \right] dx + C(\tau,T) \leq C(\tau,T),   \]
where we used Lemma $\ref{density_time}$ in the final inequality.
\end{proof}

Now this gives enough estimates to prove our $L^1(Q_{\tau,T})$ bounds on $\p_t p$. Indeed we proceed as in \cite{perthame2014hele} and use that  $|\p_t p| = kv^{k-2} |q|$ to see
\[ \iint_{Q_{\tau,T}} |\p_t p|  \leq \iint_{ \{ v < 1/2 \} \cap Q_{\tau,T} } kv^{k-2}|q|  + 2\iint_{ \{v \geq 1/2\} \cap Q_{\tau,T} } kv^{k-1}|q|  \] \[ \leq \iint_{Q_{\tau,T}} k\frac{1}{2^{k-2}} |q| + 2kv^{k-1}|q|  \leq C(\tau,T), \] where in the final bound we used Lemma $\ref{density_time}$ and $\ref{pressure_time_lemmma}$. That is we deduced
\begin{lemma} Assume that Assumptions $\ref{Assum_1}$, $\ref{Assum_2}$, and $\ref{hard_constraint_restriciton}$ are met, then for almost every $\tau>0$ we have the bound for sufficiently large $k>1$
\[ \iint_{Q_{\tau,T}} |\p_t p_k| \leq C(\tau,T). \]

\end{lemma}

\subsection{Passage To The Limit For The Density}  

 Our bounds imply $v_k$ and $p_k$ is a bounded family in $W^{1,1}(Q_{\tau,T})$ for almost every $0<\tau<T$ (so the set of such $\tau$ is dense). Hence, along a subsequence (which we still denote by $k$), we can find a $v_{\infty},p_{\infty} \in L^1(Q_{\tau,T})$ such that
\[ v_k \rightarrow v_{\infty} \text{ in } L^1(Q_{\tau,T}) \text{ and } p_k \rightarrow p_{\infty} \text{ in } L^1(Q_{\tau,T}). \] Then by using a standard diagonal argument combined with the $L^{\infty}(Q_T)$ and support control of the pressure and density Lemma $\ref{Pressure_Bounds}$, we see that we can extract a sub-sequence (which we denote by $k)$ such that $(v_k,p_k)$ has a limit in $L^1(Q_T) \times L^1(Q_T)$ . We denote the limit (along this sub-sequence) as $v_{\infty}$ and $p_{\infty}$. Now if we denote by $\varrho_{\infty} := mv_{\infty}$, then it implies $\varrho_k \rightarrow \varrho_{\infty}$ in $L^1(Q_T)$ since $m \geq \delta > 0$. Next observe from  $v_k p_k = C_k p_k^{\frac{k}{k-1}}$ where $C_k \rightarrow 1$ as $k \rightarrow \infty$ that from the pointwise convergence a.e. (we can use a further sub-sequence if needed) that
 \[ p_{\infty}(v_{\infty}-1) = 0 \text{ a.e. } \Rightarrow p_{\infty}(m-u_{\infty})=0 \text{ a.e. }. \] Finally by considering the weak form of the density equation in $(\ref{PME_Density})$ and letting $k \rightarrow \infty$ with our estimates, shows that for any test function $\varphi$ we have
 \[ \begin{cases} \iint_{Q_T} \varphi_t \varrho_{\infty} - \nabla \varphi \cdot(m \nabla p_{\infty} + \varrho_{\infty} \vec{b}) + \varphi f \varrho_{\infty}  + \int_{\R^d} \varphi(t=0)\varrho_{\infty}^0  = 0 \\ p_{\infty}(m-\varrho_{\infty}) = 0 \text{ a.e. in } Q_T \end{cases}. \] That is $(\varrho_{\infty},p_{\infty})$ is a weak solution to the density equation in $(\ref{Hele-Shaw})$ with initial data $\varrho_{\infty}^0$. This gives us the existence portion of Theorem $\ref{Compactness-Theorem}$.

 \section{Uniqueness and Comparison Principle of the Generalized Hele-Shaw Flow } \label{uniqueness_section}
 
 In this section, we consider weak solutions  $(\varrho,p)$ to  the density equation $(\ref{Hele-Shaw})$ such that $\varrho \in L^{\infty}([0,T];L^1(\R^d) \cap L^{\infty}(\R^d))$ and $p \in L^2([0,T];H^1(\R^d) \cap L^{\infty}(\R^d))$ with compact support for any fixed time and we will show uniqueness among this class of functions with the same initial density. We proceed via the Hilbert's duality method used in \cite{perthame2014hele, guillen2022hele, david2021incompressible}. Let $(\varrho_1,p_1)$ and $(\varrho_2,p_2)$ be two weak solutions of $(\ref{Hele-Shaw})$ with the same initial data, then we see that for any $\varphi \in C^{\infty}_c(\R^d \times [0,T))$ we have
 \begin{equation} \iint_{Q_T} ((\varrho_1 - \varrho_2) + (p_1 - p_2)) \left[ A (\p_t \varphi) + B \nabla \cdot(m \nabla \varphi) - A(\nabla \varphi \cdot \vec{b} - \varphi f) \right] = 0, \label{Dual-Problem} \end{equation} where
 \begin{equation}
    0 \leq  A := \frac{\varrho_1-\varrho_2}{(\varrho_1 - \varrho_2) + (p_1 - p_2)} \leq 1 \text{ and }   0 \leq  B := \frac{p_1-p_2}{(\varrho_1 - \varrho_2)  + (p_1 - p_2)} \leq 1. \label{AB}
 \end{equation} Note that these bounds on $A$ and $B$ follow from $p_i(m-\varrho_i)=0$. for $i=1,2$ (where $A:=0$ when $\varrho_1=\varrho_2$ and $B:=0$ when $p_1=p_2$). Now let $R>0$ be so large such that $\varrho_i(t)$ and $p_i(t)$ are supported in $B_R(0) \subset \R^d$ for $t \in [0,T]$. This leads us to be interested in solving the dual problem; that is if $\psi \in C_c^{\infty}(Q_T)$ we want to find a $\varphi$ such that
 \begin{equation}   \begin{cases} A (\p_t \varphi) + B \nabla \cdot(m \nabla \varphi) - A(\nabla \varphi \cdot \vec{b} - \varphi f) = A\psi \text{ in } Q_T \\
 \varphi(t=T) = 0 \text{ and } \varphi(x,t) = 0 \text{ on } \p B_R(0) \times (0,T) \label{dual_problem_1}
 \end{cases} ,
 \end{equation} where if necessary, we take $R$ larger to ensure the support of $\psi$ is contained in $B_R(0) \times [0,T]$. We also define $\varphi = 0$ for $|x|>R$. 
 
 
 Observe that $(\ref{dual_problem_1})$ is a backwards degenerate parabolic operator due to the possible degeneracy of $A=0$ and $B=0$ at some points. So to solve $(\ref{dual_problem_1})$ we use an approximation scheme. Indeed, consider compactly supported smooth approximations $A_n$ and $B_n$ such that
 \[ \begin{cases} ||A_n - A||_{L^2(Q_T)} \leq \frac{1}{n} \text{ and } \frac{1}{n} \leq A_n \leq 1  \\ 
 ||B_n - B||_{L^2(Q_T)} \leq \frac{1}{n} \text{ and } \frac{1}{n} \leq B_n \leq 1
 \end{cases}. \]   Now we note that the approximate problem for $\psi \in C^{\infty}_c(Q_T)$
 \begin{equation}
     \begin{cases}
      \p_t \varphi_n + \frac{B_n}{A_n} \left[ \nabla \cdot(m \nabla \varphi_n) \right] - (\nabla \varphi_n \cdot \vec{b} - \varphi_n f) = \psi \text{ in } Q_T  \\
      \varphi_n(t=T)=0 \text{ and } \varphi_n(x,t) = 0 \text{ on } \p B_R(0) \times (0,T)
     \end{cases}. \label{reg_dual}
 \end{equation} is a uniform backwards parabolic equation with a terminal time condition, so the equation is classically solvable. In addition, since the coefficients are smooth, the solution is smooth with a smooth solution (see \cite{vazquez2007porous}). Then using $\varphi_n$ as a test function in $(\ref{Dual-Problem})$ gives us that
 \[ \iint_{Q_T} (\varrho_1 - \varrho_2)  \psi  = \underbrace{\iint_{Q_T} \left[ (\varrho_1 - \varrho_2) + (p_1-p_2) \right] (A \frac{B_n}{A_n} - B_n + B_n  - B) [\nabla \cdot(m \nabla \varphi_n)] }_{\mathcal{I}_n}. \]
 For notational simplicity, we denote $Z:=\left[ (\varrho_1 - \varrho_2) - (p_1-p_2) \right]$ and we observe that
 \begin{equation} \mathcal{I}_n = \underbrace{\iint_{Q_T} Z \frac{B_n}{A_n}  (A - A_n) [\nabla \cdot(m \nabla \varphi_n)] }_{\mathcal{I}_{n,1}} + \underbrace{\iint_{Q_T} Z  (B_n - B) [\nabla \cdot(m \nabla \varphi_n)] }_{\mathcal{I}_{n,2}}. \label{In_Eqn} \end{equation}
 Now we collect similar estimates on the dual problem to the ones in \cite{perthame2014hele}:
 
 \begin{lemma} There are constant $C=C(T,\psi)>0$ independent of $n$ such that for all $t \in [0,T]$
 \[ ||\varphi_n(t)||_{L^{\infty}(\R^d)} \leq C, \quad ||\nabla \varphi_n(t)||_{L^2(\R^d)} \leq C, \quad ||(B_n/A_n)^{1/2} \nabla \cdot(m \nabla \varphi_n)||_{L^2(Q_T)} \leq C.  \] \label{dual_phi_ineq}
 \end{lemma}
 
 \begin{proof}
 For the first inequality, observe that the equation for $\varphi_n$ is a backwards uniformly parabolic equation, and we can compare $\varphi_n$ to $e^{\pm Ct}-1$ for large $C>0$ to obtain the bounds with $C$ independent of $n$. \\
 
 For the second bound, we multiply $(\ref{reg_dual})$ by $\nabla \cdot(m \nabla \varphi_n)$ and integrate over $Q_{t,T} = \R^d \times (t,T)$ for $0 \leq t < T$ to see that 
 \begin{equation} \underbrace{\iint_{Q_{t,T}} -\frac{m}{2}\p_t |\nabla \varphi_n|^2}_{\mathcal{I}_1} + \frac{B_n}{A_n} |\nabla \cdot(m\nabla \varphi_n)|^2 - \underbrace{(\nabla \cdot(m \nabla \varphi_n)) (\nabla \varphi_n \cdot \vec{b} - \varphi_n f)}_{\mathcal{I}_2} = \underbrace{\iint_{Q_{t,T}} \psi  (\nabla \cdot(m \nabla \varphi_n))}_{\mathcal{I}_3}. \label{dual_1} \end{equation} Now observe that from integration by parts that
 \begin{equation} \mathcal{I}_1 = -\frac{1}{2} \iint_{Q_{t,T}} \p_t(m |\nabla \varphi_n|^2) - |\nabla \varphi_n|^2 (\p_t m) = \frac{1}{2} \int_{\R^d} m(t)|\nabla \varphi_n(t)|^2dx + \frac{1}{2} \iint_{Q_{t,T}} (\p_t m) |\nabla \varphi_n|^2  \label{Dual_I1} \end{equation}
 \begin{equation} \geq  \int_{\R^d} \frac{\delta}{2}|\nabla \varphi_n(t)|^2 - C \iint_{Q_{t,T}} |\nabla \varphi_n|^2. \end{equation}  And integration by parts along with Young's Inequality combined with $m$ being bounded gives
 \begin{equation} |\mathcal{I}_3| \leq \iint_{Q_{t,T}} |m(\nabla \psi \cdot \nabla \varphi_n)| \leq C(\psi) + C \iint_{Q_{t,T}} |\nabla \varphi_n|^2. \label{dual_I3} \end{equation} For the final term, we split it
 \[ \mathcal{I}_2 = \underbrace{\iint_{Q_{t,T}} (\nabla \cdot(m \nabla \varphi_n))(\nabla \varphi_n \cdot \vec{b}) }_{\mathcal{I}_{2,1}} - \underbrace{\iint_{Q_{t,T}} (\nabla \cdot(m \nabla \varphi_n))(\varphi_n f) }_{\mathcal{I}_{2,2}}.  \] Then we have from integration by parts and Young's Inequality that combined with our first estimate of $||\varphi_n(t)||_{L^{\infty}(\R^d)} \leq C$ that
 \begin{equation} |\mathcal{I}_{2,2}| \leq \iint_{Q_{t,T}}  |m \nabla \varphi_n \cdot (f \nabla \varphi_n + \varphi_n \nabla f)|  \leq C + C \iint_{Q_{t,T}} |\nabla \varphi_n|^2 . \label{Dual_I22} \end{equation} Now to deal with $\mathcal{I}_{2,1}$ observe
 \[ -\mathcal{I}_{2,1} = \iint_{Q_{t,T}} (m \nabla \varphi_n) \cdot \nabla(\nabla \varphi_n \cdot \vec{b}) = \iint_{Q_{t,T}} m \left( \sum_{i,j=1}^{d} \frac{\p \varphi_n}{\p x_i} \frac{\p^2 \varphi_n}{\p x_i \p x_j} b_j + \sum_{i,j=1}^{d} \frac{\p \varphi_n}{\p x_i} \frac{\p \varphi_n}{\p x_j} \frac{\p b_j}{\p x_i} \right).   \] Then notice that
 \[ \iint_{Q_{t,T}} m  \sum_{i,j=1}^{d} \left| \frac{\p \varphi_n}{\p x_i} \frac{\p \varphi_n}{\p x_j} \frac{\p b_j}{\p x_i}\right|  \leq C \iint_{Q_{t,T}}  \sum_{i,j=1}^{d} \left| \frac{\p \varphi_n}{\p x_i} \frac{\p \varphi_n}{\p x_j} \right| \leq  C \iint_{Q_{t,T}} |\nabla \varphi_n|^2,  \]
 where for the final inequality we used Young's Inequality. And we also have from integration by parts on $x_j$ that
 \[ \iint_{Q_{t,T}}    \sum_{i,j=1}^{d} \frac{\p \varphi_n}{\p x_i} \frac{\p^2 \varphi_n}{\p x_i \p x_j} (m b_j)  = \iint_{Q_{t,T}} -\sum_{i,j=1}^{d} \frac{\p^2 \varphi_n}{\p x_i \p x_j} \frac{\p \varphi_n}{\p x_i} (mb_j) - \sum_{i,j=1}^{d} |\frac{\p \varphi_n}{\p x_i}|^2 \frac{\p (mb_j)}{\p x_j}.  \] Hence, we deduce that
 \[  \iint_{Q_{t,T}}   \sum_{i,j=1}^{d} \frac{\p \varphi_n}{\p x_i} \frac{\p^2 \varphi_n}{\p x_i \p x_j} (m b_j)  = -\frac{1}{2} \iint_{Q_{t,T}}  \sum_{i,j=1}^{d} |\frac{\p \varphi_n}{\p x_i}|^2 \frac{\p (mb_j)}{\p x_j}.   \] This implies that we have from Young's Inequality that
 \begin{equation} |\mathcal{I}_{2,1}| \leq C + C \iint_{Q_{t,T}} |\nabla \varphi_n|^2 . \label{DualI21} \end{equation} Using our bounds of $\mathcal{I}_1, \mathcal{I}_2$, and $\mathcal{I}_3$ (see \eqref{Dual_I1}, \eqref{DualI21}, \eqref{Dual_I22}, \eqref{dual_I3}) in \eqref{dual_1} gives
 \begin{equation} \int_{\R^d} \frac{\delta}{2} |\nabla \varphi_n(t)|^2  + \iint_{Q_{t,T}} \frac{B_n}{A_n}  |\nabla \cdot(m \nabla \varphi_n)|^2 \leq C(\psi) + C(\psi) \iint_{Q_{t,T}} |\nabla \varphi_n(t)|^2, \label{BnAnIneq} \end{equation} which by $B_n/A_n \geq 0$, lets us use the integral form of Grönwall's inequality to deduce that
 \begin{equation} \int_{\R^d} |\nabla \varphi_n(t)|^2 \leq C(\psi,T). \label{2nd_phi} \end{equation}

 To obtain the 3rd inequality in the statement of Lemma \ref{dual_phi_ineq}, we use \eqref{2nd_phi} and \eqref{BnAnIneq} to see that
 \[ ||(B_n/A_n)^{1/2} \nabla \cdot(m \nabla \varphi_n)||_{L^2(Q_{t,T})} \leq C(\psi,T) \] for all $0 \leq t \leq T$, so by letting $t \rightarrow 0$ we conclude the third inequality.
 \end{proof}
 
Now these inequalities will be  used to bound $\mathcal{I}_{n,1}$ and $\mathcal{I}_{n,2}$ from $(\ref{In_Eqn})$. Indeed, observe that since $Z$ is bounded we have by H\"older's Inequality with Lemma \ref{dual_phi_ineq} that
\[ |\mathcal{I}_{n,1}| \leq C \iint_{Q_T} \left[ (\frac{B_n}{A_n})^{1/2}|A-A_n| \right] \left[ (\frac{B_n}{A_n})^{1/2} |\nabla \cdot(m \nabla \varphi_n)| \right] \leq C ||(B_n/A_n)^{1/2} (A-A_n)||_{L^2(Q_T)}  \] 
\[ \leq \frac{Cn^{1/2}}{n} = C/n^{1/2}.  \] Also we have the following bound from Lemma \ref{dual_phi_ineq}
\[ |\mathcal{I}_{n,2}| \leq C\iint_{Q_T} |B_n-B| \left(\frac{A_n}{B_n}\right)^{1/2} \left( \frac{B_n}{A_n} \right)^{1/2} |\nabla \cdot(m \nabla \varphi_n)| \] \[ \leq C || (A_n/B_n)^{1/2} (B_n-B) ||_{L^2(Q_T)} \leq C/n^{1/2}.   \] So by letting $n \rightarrow \infty$ in \eqref{In_Eqn}, we see that for any $\psi \in C^{\infty}_c(Q_T)$ that
\[ \iint_{\R^d} (\varrho_1 - \varrho_2) \psi = 0. \] That is we have shown that $\varrho_1 = \varrho_2$ almost everywhere. Now using the density equation in $(\ref{Hele-Shaw})$ with $\varrho_1=\varrho_2$ a.e., we conclude that for all $\varphi \in C^{\infty}_c(Q_T)$
\[ \iint_{Q_T} \nabla \cdot(m \nabla \varphi)(p_1-p_2)=0.  \] By an approximation argument we may take $\varphi = p_1 - p_2$ to conclude by integration by parts that
\[ \iint_{Q_T} m |\nabla(p_1-p_2)|^2 =0, \] so from $m \geq \delta > 0$, we see this implies from $p_i$ being compactly supported that $p_1 = p_2$ a.e., so we conclude that:

\begin{theorem} [Uniqueness] There is at most one weak solution of $(\ref{Hele-Shaw})$  $(\varrho,p)$ in the regularity class $L^{\infty}( (0,T); L^1(\R^d) \cap L^{\infty}(\R^d)) \times L^2([0,T];H^1(\R^d) \cap L^{\infty}(Q_T))$ with compact support and $\varrho(x,0) = \varrho_{\infty}^{(0)}(x)$. $\label{uniqueness}$
\end{theorem}

Note that if Assumptions $\ref{Assum_1}$, $\ref{Assum_2}$, and $\ref{hard_constraint_restriciton}$ are met then our limiting density and pressure $(\varrho_{\infty},p_{\infty})$ satisfy the hypothesis of $(\ref{uniqueness})$. Hence, they are the unique solutions of the density equation in $(\ref{Hele-Shaw})$ in this regularity class with $\varrho_{\infty}(x,0)=\varrho_{\infty}^{(0)}(x)$. This immediately implies with Urysohn’s subsequence principle the following corollary:
 
\begin{corollary} Assume that Assumptions $\ref{Assum_1}$, $\ref{Assum_2}$,  and $\ref{hard_constraint_restriciton}$ are met, then the entire sequence $(\varrho_k,p_k)$ converges to $(\varrho_{\infty},p_{\infty})$ in $L^1(Q_T) \times L^1(Q_T)$.
\end{corollary}

 This concludes the proof of Theorem \ref{Compactness-Theorem}. In addition, an analogous proof of the above dual problem shows the following comparison principle for bounded domains: (see \cite{guillen2022hele} Proposition 5.1 for the case of $m \equiv 1$ with $\vec{b}=0$)
 
 \begin{theorem}[Density Comparison Principle For Limiting Equation] \label{comparison} Let $\Om \subset \R^d$ be a open and connected subset of $\R^d$ with smooth boundary and $0 \leq t_1 < t_2 < \infty$. Assume that $(\varrho_1,p_1)$ and $(\varrho_2,p_2)$ are weak solutions of $(\ref{Hele-Shaw})$ in $\Om \times (t_1,t_2)$ such that $(\varrho_i,p_i) \in $ $L^{\infty}( (t_1,t_2); L^1(\Om) \cap L^{\infty}(\Om)) \times L^2((t_1,t_2);H^1(\Om) \cap L^{\infty}(\Om))$ with compact support such that $\varrho_1(t=t_1) \geq \varrho_2(t=t_1)$ and $p_1 \geq p_2$ on $\p \Om \times [t_1,t_2]$. Then we have that
 \[ \varrho_1(x,t) \geq \varrho_2(x,t) \text{ a.e. in } \Om \times [t_1,t_2]. \]
 
 \end{theorem}

 \section{The Complimentary Relationship on the Pressure} Now we focus on deriving the complimentary relationship. That is we will show that $p_{\infty}$ is a weak solution to the pressure equation in $(\ref{Hele-Shaw})$. To derive the complimentary relationship we modify the arguments in \cite{david2021incompressible} to obtain a uniform $L^4(Q_{\tau,T})$ bounds on $|\nabla p_k|$. To achieve such bounds, we first prove the following technical lemma: 
 
 \begin{lemma}  Let $f \in C^{\infty}_c(\R^d)$ then there is a constant $C>0$ independent of $f$ such that
 \[ \int_{\R^d} |\nabla f|^4 \leq C \int_{\R^d} |f|^2 |\nabla f|^2 + C \int_{\R^d} \frac{|f|^2}{m}|\nabla \cdot(m \nabla f)|^2 + C\sum_{i,j=1}^{d} \int_{\R^d}   m |f|^2|\p^2_{x_i x_j} f|^2.  \]

 \end{lemma}
 
 \begin{proof} Observe that from integration by parts that
 \[ \int_{\R^d} |\nabla f|^4 =  \int_{\R^d} \left( |\nabla f|^2 \nabla f \right) \cdot \nabla f  =  -\underbrace{\int_{\R^d} f \Delta f  |\nabla f|^2}_{\mathcal{I}_1} -  \underbrace{\int_{\R^d} f (\nabla f) \cdot \nabla(|\nabla f|^2)}_{\mathcal{I}_2}.  \] Expaning out $\mathcal{I}_2$ shows us that
 \[ \mathcal{I}_2 = \sum_{i,j=1}^{d} 2\int_{\R^d} f(\p_{x_j} f)(\p_{x_i} f) (\p^2_{x_i x_j} f). \]  Hence, by applying Young's Inequality, we have that
 \[ |\mathcal{I}_2| \leq \frac{1}{4} \int_{\R^d} |\nabla f|^4 + 4 \sum_{i,j=1}^{d} \int_{\R^d} f^2 |\p^2_{x_i x_j} f|^2. \] Now using that $\frac{1}{m} \nabla \cdot(m \nabla f) = \nabla \lambda \cdot \nabla f + \Delta f$ implies
 \[ \mathcal{I}_1 = \int_{\R^d} \frac{f}{m} \nabla \cdot(m \nabla f) |\nabla f|^2 - \int_{\R^d} f(\nabla \lambda \cdot \nabla f) |\nabla f|^2. \] Now by applying Young's Inequality and Cauchy-Schwarz we have that
 \[ |\mathcal{I}_1| \leq \frac{1}{2} \int_{\R^d} \frac{|f|^2}{m^2} |\nabla \cdot(m \nabla f)|^2 + \frac{2}{3} \int_{\R^d} |\nabla f|^4  + C \int_{\R^d} |f|^2 |\nabla f|^2.   \] By using $m \geq \delta >0$ and combining terms we obtain the desired bound.
 \end{proof}
 
 This gives us the following $L^4(Q_T)$ control of the pressure gradient by an approximation argument due to our $L^{\infty}(Q_T)$ control of $p_k$ and $L^2(Q_T)$ control of $\nabla p_k$ (see Lemma \ref{Pressure_Bounds} and Lemma \ref{L2_Pressure_Gradient}): 
 
\begin{corollary} If Assumptions $\ref{Assum_1}$, $\ref{Assum_2}$, and $\ref{hard_constraint_restriciton}$ are met, then there is a constant $C=C(T)$ such that
\begin{equation}
    \iint_{Q_T} |\nabla p_k|^4 \leq C + C \iint_{Q_T} \frac{p_k}{m} |\nabla \cdot(m \nabla p_k)|^2 + C \sum_{i,j=1}^{d} \iint_{Q_T} m p_k |\p^2_{x_i x_j} p_k|^2. \end{equation}
    \label{L4_Pressure_Bounds1}
\end{corollary}

 
 Now we control the terms arising in Corollary $\ref{L4_Pressure_Bounds1}$ using \eqref{PME_pressure}.
 
 \begin{lemma} If Assumptions $\ref{Assum_1}$, $\ref{Assum_2}$, $\ref{hard_constraint_restriciton}$, and $\ref{complimentary_assumpution}$ are met, then there is a constant $C$ and $\tilde{C}$ such that for $k>19/3$ we have
 \[   \frac{1}{12} \sum_{i,j=1}^{d} \iint_{Q_T} mp_k|\p^2_{x_i x_j} p_k|^2 + (k-\frac{19}{3}) \iint_{Q_T} \frac{p_k}{m} |\nabla \cdot(m \nabla p_k) + mF|^2 \leq C \int_{\R^d} |\nabla p_k(x,0)|^2 dx + C \leq \tilde{C}. \] \label{L4_2}
 \end{lemma}

 \begin{remark} This estimate with our smooth approximations of $p_k$ constructed in Section \ref{weak_mPME} implies that $p_k \in H^2_{loc}( \{p_k>0\} ).$ Then we interpret $p_k |\p^2_{x_i x_j} p_k|^2$ classically on $\{p_k>0\}$ and as 0 on $\{p_k=0\}$.
 \end{remark}

\begin{proof} We will drop the subscript $k$ for the rest of this proof. Multiplying the pressure equation \eqref{PME_pressure} by $\omega := \nabla \cdot(m \nabla p) + mF$ and integrating over $Q_T$ yields
\begin{equation} \underbrace{\iint_{Q_T} \omega \p_t p}_{\mathcal{I}_1} =  \underbrace{\iint_{Q_T} \omega |\nabla p|^2}_{\mathcal{I}_2}  + \underbrace{\iint_{Q_T} \omega(\nabla p \cdot \vec{b})}_{\mathcal{I}_3} + \iint_{Q_T} (k-1) \frac{p}{m} |\omega|^2. \label{all_term}  \end{equation} We integrate each term separately. Observe
\[ \mathcal{I}_1 = \iint_{Q_T} \left[\nabla \cdot(m \nabla p)\right] \p_t p + \iint_{Q_T} mF (\p_t p) = \iint_{Q_T} -m \nabla p \cdot \p_t(\nabla p) + mF(\p_t p)  \]
\[ = \iint_{Q_T} -\frac{1}{2} \p_t(m |\nabla p|^2) + \frac{1}{2} m_t |\nabla p|^2 + mF(\p_t p). \]
So now using $mF (\p_t p) = \p_t(p mF) - p\p_t(mF)$ and integrating gives
\[ = \int_{\R^d} \frac{1}{2} \left[ m(x,0)|\nabla p(x,0)|^2 - m(x,T)|\nabla p(x,T)|^2 + m(x,T)F(x,T)p(x,T)-m(x,0)F(x,0)p(x,0) \right] + ... \] \[ + \iint_{Q_T} \frac{1}{2} m_t|\nabla p|^2 - p\p_t (mF).  \] So our $L^{\infty}(Q_T)$ bounds on $p_k$ (Lemma \ref{Pressure_Bounds}) and $L^2(Q_T)$ bounds on $\nabla p_k$ (Lemma \ref{L2_Pressure_Gradient}) implies
\begin{equation} \boxed{\mathcal{I}_1 \leq C \int_{\R^d} |\nabla p(x,0)|^2 + C.} \label{Final_I1} \end{equation}\\
Now we deal with $\mathcal{I}_2$. Observe that
\begin{equation} \mathcal{I}_2 = \underbrace{\iint_{Q_T} \nabla \cdot(m \nabla p) |\nabla p|^2}_{\mathcal{I}_{2,1}} + \iint_{Q_T} mF |\nabla p|^2, \label{I2_eq1} \end{equation} so integrating by parts twice gives
\[ \mathcal{I}_{2,1} = \iint_{Q_T}p \nabla \cdot(m \nabla(|\nabla p|^2)) = \underbrace{\iint_{Q_T} p (\nabla m \cdot \nabla(|\nabla p|^2))}_{\mathcal{J}_1} + \underbrace{\iint_{Q_T}  mp \Delta(|\nabla p|^2)}_{\mathcal{J}_2}. \] Now using that $\Delta(|\nabla p|^2) = 2\sum_{i,j=1}^{d} |\p^2_{x_i x_j} p|^2 + 2 \nabla p \cdot \nabla(\Delta p)$ gives
\[ \mathcal{J}_2 = 2\iint_{Q_T} mp \sum_{i,j=1}^{d} |\p^2_{x_i x_j} p|^2 + \underbrace{2\iint_{Q_T} mp(\nabla p \cdot \nabla (\Delta p))}_{\mathcal{J}_{2,1}}, \] and integration by parts also implies
\begin{equation} \mathcal{J}_{2,1} = -2 \iint_{Q_T} \nabla \cdot(p (m \nabla p)) \Delta p = -2 \iint_{Q_T} m|\nabla p|^2 \Delta p + p \Delta p \left[ \nabla \cdot(m \nabla p) \right]. \label{J21_eq1} \end{equation} Now we want to simplify the expression for $\mathcal{J}_{2,1}$ using our previous expressions. First notice from the definition of $\mathcal{I}_{2,1}$ in \eqref{I2_eq1} that
\begin{equation} \mathcal{I}_{2,1} -  \iint_{Q_T} [\nabla m \cdot \nabla p] |\nabla p|^2  =  \iint_{Q_T} m \Delta p |\nabla p|^2 \label{newI21}. \end{equation} So we obtain from \eqref{newI21} and \eqref{J21_eq1} that
\begin{equation} \mathcal{J}_{2,1} = -2 \mathcal{I}_{2,1} + 2 \iint_{Q_T} [\nabla m \cdot \nabla p]|\nabla p|^2 - p \Delta p[\nabla \cdot(m \nabla p)]. \label{J21_1} \end{equation}  Observe that by unpacking our definitions that
\[ \mathcal{I}_{2,1} = \mathcal{J}_1 + 2 \iint_{Q_T} mp \sum_{i,j=1}^{d} |\p^2_{x_i x_j} p|^2 + \mathcal{J}_{2,1}, \] so by using \eqref{J21_1}, we see that
\begin{equation}  \mathcal{I}_{2,1} = \frac{1}{3} \mathcal{J}_1 + \frac{2}{3} \iint_{Q_T} mp \sum_{i,j=1}^d |\p^2_{x_i x_j} p|^2 +  \iint_{Q_T} \frac{2}{3} [\nabla m \cdot \nabla p]|\nabla p|^2 - \frac{2}{3} p \Delta p[\nabla \cdot(m \nabla p)]. \label{newI2} \end{equation}
Now our goal is to compute lower bounds of $\mathcal{I}_{2,1}$. To deal with the $\iint_{Q_T} [\nabla m \cdot \nabla p] |\nabla p|^2$ term, we first obtain bounds on $\iint_{Q_T} p \Delta p(\nabla p \cdot \nabla m)$. Integrating by parts gives that
\begin{equation} \iint_{Q_T} p\Delta p (\nabla p \cdot \nabla m) = -\iint_{Q_T} |\nabla p|^2[\nabla p \cdot \nabla m] - \underbrace{\iint_{Q_T} p [\nabla p \cdot \nabla( \nabla p \cdot \nabla m)]}_{\mathcal{J}_{2,2}}. \label{L4_I2} \end{equation} Note that $p[ \nabla p \cdot \nabla(\nabla p \cdot \nabla m)] = p[ \nabla p^T D^2p \nabla m + \nabla p^T D^2m \nabla p]  $ so Young's Inequality with $\e$  gives
\[ |\mathcal{J}_{2,2}| \leq \sum_{i,j=1}^{d} \frac{\delta}{2} \iint_{Q_T} p |\p^2_{x_i x_j} p|^2  + C \iint_{Q_T} p|\nabla p|^2  + \iint_{Q_T} p |(\nabla p)^T (D^2m) (\nabla p)| \]
\[ \leq \sum_{i,j=1}^{d} \frac{1}{2} \iint_{Q_T} mp |\p^2_{x_i x_j} p|^2  + \tilde{C} \iint_{Q_T} p|\nabla p|^2   \]
\begin{equation} \leq \sum_{i,j=1}^{d} \frac{1}{2} \iint_{Q_T} mp |\p^2_{x_i x_j} p|^2  + C. \label{J22_Bounds} \end{equation} In this final inequality we used our $L^2(Q_T)$ bounds on $\nabla p$ (Lemma \ref{L2_Pressure_Gradient}) and $L^{\infty}(Q_T)$ bounds on $p$ (Lemma \ref{Pressure_Bounds}).\\

 And to deal with the $[\nabla m \cdot \nabla p]|\nabla p|^2$ term in $\mathcal{J}_{2,1}$, we observe from \eqref{L4_I2} and \eqref{J22_Bounds} that
\begin{equation} \iint_{Q_T} |\nabla p|^2[\nabla p \cdot \nabla m] \geq - C -\iint_{Q_T} p \Delta p(\nabla p \cdot \nabla m) - \frac{1}{2} \sum_{i,j=1}^{d} \iint_{Q_T} mp|\p^2_{x_i x_j} p|^2 . \label{pm_term} \end{equation} 
Now using \eqref{pm_term} in \eqref{newI2} gives us
\[ \mathcal{I}_{2,1} \geq -C + \frac{1}{3} \mathcal{J}_1 + \frac{1}{3} \left[\iint_{Q_T}  mp \sum_{i,j=1}^{d} |\p^2_{x_i x_j} p|^2 \right] - \frac{2}{3}  \left[ \iint_{Q_T} p\Delta p[\nabla \cdot(m \nabla p)] + p \Delta p(\nabla p \cdot \nabla m) \right] . \] Now we use that
\[ p \Delta p[\nabla \cdot(m \nabla p)] + p \Delta p (\nabla p \cdot \nabla m) = mp |\Delta p|^2 + 2p\Delta p (\nabla p \cdot \nabla m)   \]
\[ \leq mp |\Delta p|^2 + 2p \Delta p(\nabla p \cdot \nabla m) +  \frac{p}{m} |\nabla p \cdot \nabla m|^2 = \frac{p}{m} |\nabla \cdot(m \nabla p)|^2. \] Hence, we have that
\begin{equation} \mathcal{I}_{2,1} \geq -C + \frac{1}{3}  \mathcal{J}_1  + \frac{1}{3} \iint_{Q_T} mp \sum_{i,j=1}^{d} |\p^2_{x_i x_j} p|^2  - \left[ \frac{2}{3} \iint_{Q_T}  \frac{p}{m} |\nabla \cdot(m \nabla p)|^2 \right]. \label{I21_Part2} \end{equation} We use integration by parts to see
\[ \mathcal{J}_1 = -\iint_{Q_T} \left([\nabla p \cdot \nabla m] + p \Delta m\right) |\nabla p|^2. \] Hence, by using Young's Inequality with $\e$ and our $L^2(Q_T)$  bounds on $\nabla p$ (Lemma \ref{L2_Pressure_Gradient}) and $L^{\infty}(Q_T)$ bounds on the pressure (Lemma \ref{Pressure_Bounds}), we see that for $\e>0$
\[ |\mathcal{J}_1| \leq \e \iint_{Q_T}  |\nabla p|^4 + C(\e) \iint_{Q_T} ||\nabla m||^2_{\infty} |\nabla p|^2  + C \iint_{Q_T} |\nabla p|^2  \]
\[ \leq \e \iint_{Q_T} |\nabla p|^4 + C(\e).\] Now using Corollary \ref{L4_Pressure_Bounds1} with an approximation argument and a fixed $\e>0$ chosen such that for the constant $C$ in from Corollary \ref{L4_Pressure_Bounds1} satisfies $\e C \leq 1/2$ gives
\begin{equation} |\mathcal{J}_1| \leq C + \frac{1}{2} \iint_{Q_T} \frac{p}{m} |\nabla \cdot(m \nabla p)|^2 + \frac{1}{2} \sum_{i,j=1}^{d} mp |\p^2_{x_i x_j} p|^2. \label{J1_Bounds} \end{equation} Then by using \eqref{J1_Bounds} in \eqref{I21_Part2} we obtain
\[ \mathcal{I}_{2,1} \geq \frac{1}{6} \sum_{i,j=1}^{d} \iint_{Q_T} mp |\p^2_{x_i x_j} p|^2  - \frac{5}{6} \iint_{Q_T} \frac{p}{m} |\nabla \cdot(m \nabla p)|^2  - C. \] Now using that
\[ |\nabla \cdot(m \nabla p)|^2 \leq 4|\nabla \cdot(m \nabla p) + mF|^2 + 4|mF|^2, \] and that $F$ is bounded due to our assumptions implies
\[ \mathcal{I}_{2,1} \geq \frac{1}{6} \sum_{i,j=1}^{d} \iint_{Q_T} mp|\p^2_{x_i x_j} p|^2 - \frac{10}{3} \iint_{Q_T} \frac{p}{m} |\nabla \cdot(m \nabla p)+mF|^2  - C. \]
So now using \eqref{I2_eq1} and our $L^2(Q_T)$ bound on $\nabla p$ (Lemma \ref{L2_Pressure_Gradient}), we deduce that

\begin{equation} \boxed{\mathcal{I}_2 \geq -C + \frac{1}{6} \sum_{i,j=1}^{d} \iint_{Q_T}  mp  |\p^2_{x_i x_j} p|^2 - \frac{10}{3} \iint_{Q_T} \frac{p}{m} |\nabla \cdot(m \nabla p) + mF|^2.}  \label{final_I2} \end{equation}  

So it remains to deal with $\mathcal{I}_3$ to obtain the desired bounds. Observe that
\begin{equation} \mathcal{I}_3 = \underbrace{\iint_{Q_T} (\nabla p \cdot \vec{b}) (\nabla \cdot(m \nabla p))}_{\mathcal{I}_{3,1}} + \underbrace{\iint_{Q_T} mF (\nabla p \cdot \nabla \vec{b})}_{\mathcal{I}_{3,2}}. \label{I3} \end{equation} Now Cauchy-Schwarz and our $L^1(Q_T)$ control of $\nabla p_k$ (Lemma \ref{Pressure_Bounds} and Lemma \ref{L2_Pressure_Gradient}) implies
\begin{equation} \mathcal{I}_{3,2} \geq -C \iint_{Q_T} |\nabla p|  \geq -\tilde{C}. \label{I32} \end{equation} Now we compute from integration by parts that
\[ \mathcal{I}_{3,1} = -\iint_{Q_T} p \nabla \cdot(\vec{b} \left[\nabla \cdot(m \nabla p) \right]) = -\underbrace{\iint_{Q_T} p (\nabla \cdot \vec{b}) [\nabla \cdot(m \nabla p)]}_{\mathcal{K}_1} - \underbrace{\iint_{Q_T} p (\vec{b} \cdot \nabla \left[ \nabla \cdot(m \nabla p) \right])}_{\mathcal{K}_2}. \] We apply Young's Inequality to see that
\begin{equation} |\mathcal{K}_1| \leq \frac{1}{2 ||m||_{\infty}} \iint_{Q_T} p |\nabla \cdot(m \nabla p)|^2 + \frac{||m||_{\infty}}{2} \iint_{Q_T} p|\nabla \cdot \vec{b}|^2 \leq C + \frac{1}{2} \iint_{Q_T} \frac{p}{m} |\nabla \cdot(m \nabla p)|^2. \label{K1} \end{equation} Now we compute to see that
\[ \mathcal{K}_2 = \underbrace{\sum_{i,j=1}^{d} \iint_{Q_T} p b_i\left[ (\p^2_{x_i x_j} m)(\p_{x_j} p) \right]}_{\mathcal{K}_{2,1}} + \underbrace{\sum_{i,j=1}^{d} \iint_{Q_T} p b_i \left[ (\p_{x_j} m)(\p^2_{x_i x_j} p) \right]}_{\mathcal{K}_{2,2}} + \underbrace{\sum_{i,j=1}^{d} \iint_{Q_T} p (mb_i) \left[  \p^3_{x_j x_i x_j} p.  \right]}_{\mathcal{K}_{2,3}} + \]
\[ + \iint_{Q_T} (\vec{b} \cdot \nabla m) \Delta p. \] The last term can be bounded by integration by parts:
\begin{equation} \left| \iint_{Q_T} (\vec{b} \cdot \nabla m) \Delta p \right| = \left| \iint_{Q_T} - \nabla (\vec{b} \cdot \nabla m) \cdot \nabla p \right| \leq C \iint_{Q_T} |\nabla p|  \leq C. \label{K2_Last} \end{equation} Now we bound the other terms. Observe that from Young's Inequality that
\begin{equation} |\mathcal{K}_{2,1}| \leq \frac{1}{2} \sum_{i,j=1}^{d} \iint_{Q_T} p |\p_{x_i} p|^2 + p |b_i|^2|\p^2_{x_i x_j} m|^2 \leq C(T), \label{K21} \end{equation} where we used our $L^2(Q_T)$ bounds on $\nabla p$ (Lemma \eqref{L2_Pressure_Gradient}) and our $L^{\infty}(Q_T)$ bounds on $p$ (Lemma \eqref{Pressure_Bounds}). Then by using Young's Inequality again
\begin{equation} |\mathcal{K}_{2,2}| \leq \iint_{Q_T} \frac{\delta}{24} \sum_{i,j=1}^{d} p |\p^2_{x_i x_j} p|^2 + \frac{12}{\delta} \sum_{i,j=1}^{d}  p|b_i|^2|\p_{x_j} m|^2  \leq  \frac{1}{24} \iint_{Q_T} \sum_{i,j=1}^{d} mp|\p^2_{x_i x_j}p|^2  + C(T).  \label{K22} \end{equation}
Now we deal with $\mathcal{K}_{2,3}$. Let us define $\vec{c}:=m \vec{b} = (c_1,c_2,...,c_d)^T$ and integrate by parts to see
\begin{equation} \mathcal{K}_{2,3} = -\underbrace{\sum_{i,j=1}^{d} \iint_{Q_T} c_i (\p_{x_j} p)(\p^2_{x_i x_j} p)}_{\mathcal{K}_{2,4}} - \underbrace{\sum_{i,j=1}^{d} \iint_{Q_T} p (\p_{x_j}c_i) (\p^2_{x_j x_i} p)}_{\mathcal{K}_{2,5}}. \end{equation} By Young's Inequality we have that
\begin{equation} |\mathcal{K}_{2,5}| \leq \iint_{Q_T} \frac{\delta}{24} \sum_{i,j=1}^{d} p|\p^2_{x_i x_j} p|^2 + \frac{12}{\delta}p |\p_{x_j} c_i|^2 \leq  \frac{1}{24} \sum_{i,j=1}^{d} \iint_{Q_T} mp |\p^2_{x_i x_j} p|^2  + C(T).  \label{K25}  \end{equation} Now we observe that
\[ \mathcal{K}_{2,4} = \frac{1}{2} \iint_{Q_T} \vec{c} \cdot  \nabla(|\nabla p|^2) = \frac{1}{2} \iint_{Q_T} \nabla \cdot( \vec{c}  |\nabla p|^2) -  (\nabla \cdot \vec{c}) |\nabla p|^2  \]
\[ = -\frac{1}{2} \iint_{Q_T} (\nabla \cdot \vec{c})|\nabla p|^2. \] Hence, we have from our $L^2(Q_T)$ control of $\nabla p$ (Lemma \ref{L2_Pressure_Gradient})
\begin{equation} |\mathcal{K}_{2,4}| \leq C(T). \label{K24} \end{equation} So now combining our inequalities on $\mathcal{K}_{1}$ and $\mathcal{K}_{2}$ (see \eqref{K1}, \eqref{K21}, \eqref{K2_Last}, \eqref{K22}, \eqref{K24}, \eqref{K25}) gives
\[ |\mathcal{I}_{3,1}| \leq C(T) + \frac{1}{12} \sum_{i,j=1}^{d} \iint_{Q_T} mp|\p^2_{x_i x_j} p|^2 + \frac{1}{2} \iint_{Q_T} \frac{p}{m}|\nabla \cdot(m \nabla p)|^2.   \] Now again using
\[ |\nabla \cdot(m \nabla p)|^2 \leq 4|\nabla \cdot(m \nabla p) + mF|^2 + 4|mF|^2, \] we deduce that
\begin{equation} |\mathcal{I}_{3,1}| \leq C(T) + \frac{1}{12} \sum_{i,j=1}^{d} \iint_{Q_T} mp|\p^2_{x_i x_j} p|^2 + 2 \iint_{Q_T} \frac{p}{m}|\nabla \cdot(m \nabla p) + mF|^2. \label{I31}  \end{equation} So by using \eqref{I31} and \eqref{I32} in \eqref{I3} we arrive at
\begin{equation} \boxed{\mathcal{I}_3 \geq -C(T) - \frac{1}{12} \sum_{i,j=1}^{d} \iint_{Q_T} mp|\p^2_{x_i x_j} p|^2 - 2 \iint_{Q_T} \frac{p}{m}|\nabla \cdot(m \nabla p) + F|^2.}  \label{Final_I3} \end{equation}

That is we have from \eqref{Final_I1}, \eqref{final_I2}, \eqref{Final_I3}, \eqref{all_term}, and Assumption (\ref{complimentary_assumpution}) that
\[  \frac{1}{12} \sum_{i,j=1}^{d} \iint_{Q_T} mp|\p^2_{x_i x_j} p|^2 + (k-\frac{19}{3}) \iint_{Q_T} \frac{p}{m} |\nabla \cdot(m \nabla p) + F|^2 \leq C \int_{\R^d} |\nabla p(x,0)|^2  + C \leq \tilde{C}.  \]  
 \end{proof}

Now by combining Lemma $\ref{L4_2}$ with Lemma $\ref{L4_Pressure_Bounds1}$ we obtain uniform $L^4(Q_T)$ bounds on $\nabla p$ for $k \geq 7$ due to an approximation argument. 

\begin{corollary} [$L^4(Q_T)$ bounds on $|\nabla p|$]  If Assumptions $\ref{Assum_1}$, $\ref{Assum_2}$, $\ref{hard_constraint_restriciton}$, and $\ref{complimentary_assumpution}$ are met, then for $k \geq 7$ there is a constant $C=C(T)$ such that
\[ \iint_{Q_T} |\nabla p|^4  \leq C. \]
\end{corollary}

Now we have enough compactness for strong $L^2(Q_T)$ convergence of $|\nabla p_k|$. 
\begin{lemma} Assume that Assumptions $\ref{Assum_1}$, $\ref{Assum_2}$,  $\ref{hard_constraint_restriciton}$, and $\ref{complimentary_assumpution}$ are met, then $\nabla p_k \rightarrow \nabla p_{\infty}$ in $L^r(Q_{\tau,T})$ for any $\tau>0$ and $r \in [1,4)$. \label{Aubin-Lions}
\end{lemma}

\begin{proof} This proof is essentially identical to the proof of Lemma 3.5 in \cite{david2021incompressible}, so we omit the proof.
\end{proof}

Now we use \eqref{PME_pressure} to obtain $L^2(Q_T)$ convergence of $\nabla p_k$.

\begin{theorem}Assume that Assumptions $\ref{Assum_1}$, $\ref{Assum_2}$,  $\ref{hard_constraint_restriciton}$, and $\ref{complimentary_assumpution}$ are met then $\nabla p_k \rightarrow \nabla p_{\infty}$ in $L^2(Q_T)$.
\end{theorem}

\begin{proof} By Lemma $\ref{Aubin-Lions}$ we know that $\nabla p_k \rightarrow \nabla p_{\infty}$ in $L^2(Q_{\tau,T})$ for any $\tau>0$. So it remains to control $||\nabla p_k||_{L^2(Q_{\tau})}$ for small $\tau>0$. Integrating \eqref{PME_pressure} over $Q_{\tau}$ gives
\[ (k-2) \iint_{Q_\tau}  |\nabla p_k|^2 = \int_{\R^d} p_k(x,0)-p_k(x,\tau) + \iint_{Q_\tau} \nabla p_k \cdot (\vec{b} + (k-1) p_k \nabla \lambda) + (k-1)p_k F . \]  By applying Young's Inequality with $\e>0$ we arrive at
\[ (k-2) \iint_{Q_{\tau}} |\nabla p_k|^2 \leq \int_{\R^d} p_k(x,0)  + \iint_{Q_{\tau}} |\nabla p_k|^2 +  (k-1)\iint_{Q_{\tau}} \e ||p_k||^2_{L^{\infty}(Q_T)}  |\nabla p_k|^2 + C| (\Omega \times [0,T]) \cap Q_{\tau}|  +   \]
\[ + C(\e)(k-1) \iint_{Q_{\tau}} p_k, \] where $\Om \times [0,T] =: \Om_T$ is a compact set such that $\text{supp}(p_k) \subset \Om_T$ for all $k>1$ and $|A|$ is the Lebesgue measure of $A \subset \R^{d+1}$. Choose $\e>0$ so small such that $\e ||p_k||_{L^{\infty}(Q_T)} ||\nabla \lambda||_{\infty} \leq \frac{1}{2}$ for all $k$, we arrive at
\[  \iint_{Q_{\tau}} |\nabla p_k|^2 \leq \frac{2}{k-5}\int_{\R^d} p_k(x,0)  +  C| \Om_T \cap Q_{\tau}|. \] So for $k \geq 6$, we have from Lemma \ref{Pressure_Bounds} that
\begin{equation} \iint_{Q_{\tau}} |\nabla p_k|^2 \leq \frac{2C}{k-5} + C | \Om_T \cap Q_{\tau}|,  \label{small_L2_press} \end{equation} where $|A|$ is the Lebesgue measure of $A \subset \R^{d+1}$. Note that this constant $C$ is independent of both $\tau$ and $k$. \\

Now if $\e'>0$ is given, there exists a $K'$ sufficiently large with $\tau$ sufficiently small and $\nabla p_{\infty} \in L^2(Q_T)$ such that
\[ \iint_{Q_{\tau}} |\nabla p_k|^2 + |\nabla p_{\infty}|^2 \leq \e'/2 \] for all $k \geq K'$  thanks to \eqref{small_L2_press}. Now by making $K'$ even larger if necessary, Lemma \ref{Aubin-Lions} shows us that $\iint_{Q_{\tau,T}} |\nabla p_k - \nabla p_{\infty}|^2 \leq \e'/2$ for all $k \geq K'$. So now from the triangle inequality, we obtain
\[ \iint_{Q_T} |\nabla p_k - \nabla p_{\infty}|^2  \leq \e' \] for $k \geq K'$. \end{proof}

Now we have enough compactness to show that $p_{\infty}$ is a weak solution to the pressure equation in \eqref{Hele-Shaw}. Indeed, this follows from similar arguments as in \cite{perthame2014hele} Lemma 2.5 that showed when $m \equiv 1$ the complimentary relationship is equivalent to $L^2(Q_T)$ convergence of $\nabla p_k$. Or we could proceed as in \cite{david2021free} and take the limit of the weak form of the pressure equation \eqref{PME_pressure} to conclude the proof of Theorem \ref{Complimentary-Theorem}.

\section{Velocity Law} 

\subsection{Velocity Law} \label{Velocity_Law}

Now we aim to characterize the movement of $\{p_{\infty}(t)>0\}$. We will show the velocity law \eqref{velocity_law} in the viscosity sense. Our arguments and definitions on the velocity law and classical barriers closely follow the ones in \cite{guillen2022hele}.

\begin{definition} [Classical Super-Solutions] \label{super-solution} Let $B$ be any space ball in $\R^d$ and $0 \leq t_1 < t_2 < \infty$ and $\phi \in C_c(\overline{B} \times [t_1,t_2])$ be non-negative. We also assume that $\phi \in  C^2(\{\overline{\phi>0\}})$ and that $\p \{\phi>0\}$ is $C^2$ in space and $C^1$ in time.  Then we choose a $\rho^E \in C^{1,1}(\{\phi=0\})$ such that
\[ \p_t \rho^E \geq  \nabla \cdot(\rho^E \vec{b}) + f\rho^E \text{ in } \{\phi=0\} \] with $\rho^E \leq m$  on $\{\phi=0\}$. And we also assume that
\begin{equation}
\begin{cases}
-\nabla \cdot(m \nabla \phi) \geq \nabla \cdot(m \vec{b})-m_t + mf \text{ in } \{\phi>0\} \\
(m-\rho^E)V_{\phi} \geq  \left[ - m \nabla \phi - \vec{b}(m-\rho^E) \right] \cdot \vec{\nu} \text{ on } \p \{\phi>0\}
\end{cases}
\end{equation}
where $V_{\phi}$ is the normal velocity of $\p \{\phi(t)>0\}$ and $\vec{\nu}$ is the outward normal of $\p \{\phi(t)>0\}$. Then we define
\[ \rho := m \chi_{ \{\phi>0\} } +\rho^E \chi_{ \{\phi=0\} }. \] Then the pair $(\rho,\phi)$ is called a classical super-solution.
\end{definition}

\begin{lemma} If the pair $(\rho,\phi)$ is a classical super-solution, then they are also a weak super-solution of the limiting density equation $(\ref{Hele-Shaw})$ in $B \times (t_1,t_2)$ with initial data $\rho_1 := \rho(t=t_1)$. In particular, for all $\varphi \in C_{c}^{\infty}(B \times [t_1,t_2))$ that is non-negative we have 
\[ \iint_{B \times [t_1,t_2] } - \varphi_t \rho + \nabla \varphi \cdot(m \nabla \phi) \geq \iint_{B \times [t_1,t_2]} \varphi f \rho - \nabla \varphi \cdot (\rho \vec{b}) + \int_{B} \varphi(t=t_1) \rho_1 \]
\end{lemma}

\begin{proof} The proof is essentially identical to the proof of Proposition 7.1 in \cite{guillen2022hele}, so we omit the details.
\end{proof}

And analogously, we have the following definition for classical sub-solutions.

\begin{definition}  [Classical Sub-Solutions] \label{sub-solution} Let $B$ be any space ball and $\R^d$ and $0 \leq t_1 < t_2 < \infty$ and $\phi \in C_c(\overline{B} \times [t_1,t_2])$ be non-negative. We also assume that $\phi \in  C^2(\{\overline{\phi>0\}})$ and that $\p \{\phi>0\}$ is $C^2$ in space and $C^1$ in time. Then we choose a $\rho^E \in C^{1,1}(\{\phi=0\})$ such that
\[ \p_t \rho^E \leq  \nabla \cdot(\rho^E \vec{b}) + f\rho^E \text{ in } \{\phi=0\} \] and $\rho^E \leq m$ on $\{\phi=0\}$. And we also assume that
\begin{equation}
\begin{cases}
-\nabla \cdot(m \nabla \phi) \leq \nabla \cdot(m \vec{b})-m_t + mf \text{ in } \{\phi>0\} \\
(m-\rho^E)V_{\phi} \leq  \left[ - m \nabla \phi - \vec{b}(m-\rho^E) \right] \cdot \vec{\nu} \text{ on } \p \{\phi>0\}
\end{cases}
\end{equation}
where $V_{\phi}$ is the normal velocity of $\p \{\phi(t)>0\}$ and $\vec{\nu}$ is the outward normal of $\p \{\phi(t)>0\}$.  Now if we define
\[ \rho := m \chi_{ \{\phi>0\} } +\rho^E \chi_{ \{\phi=0\} }, \] then the pair $(\varrho,\phi)$ is called a classical sub-solution.   \end{definition}

Similarly we have the following relation between classical sub-solutions and weak sub-solutions of the limiting density equation $(\ref{Hele-Shaw})$.

\begin{lemma}
If the pair $(\rho,\phi)$ is a classical sub-solution, then they are a weak sub-solution of the limiting equation $(\ref{Hele-Shaw})$ in $B \times (t_1,t_2)$ with initial data $\rho_1 := \rho(t=t_1)$. In particular, for all $\varphi \in C_{c}^{\infty}(B \times [t_1,t_2))$ that is non-negative we have 
\[ \iint_{B \times [t_1,t_2] } - \varphi_t \rho + \nabla \varphi \cdot(m \nabla \phi) \leq \iint_{B \times [t_1,t_2]} \varphi f \rho - \nabla \varphi \cdot (\rho \vec{b}) + \int_{B} \varphi(t=t_1) \rho_1 \]
\end{lemma}

Recall that the limiting pair $(\varrho_{\infty},p_{\infty})$ is a weak solution to $(\ref{Hele-Shaw})$ in the sub-domain $B \times (t_1,t_2)$ with the initial density condition $\varrho_{\infty}(t=t_1)$ and pressure boundary conditions $p_{\infty}$ on $\p B \times [t_1,t_2]$. So we deduce from the weak density comparison principle Lemma $\ref{comparison}$ for the limiting density equation that we have
\begin{proposition} Let the pair $(\rho,\phi)$ be a classical super solution in $B \times (t_1,t_2)$, then if
\begin{enumerate}
    \item $\varrho_{\infty}(t=t_1) \leq \rho(t=t_1)$ \text{ on } B 
    \item $p_{\infty} \leq \phi$ \text{ on } $\p B \times [t_1,t_2]$ ,
\end{enumerate} then we have that $\varrho_{\infty} \leq \rho$ on $B \times [t_1,t_2].$
\end{proposition}

So in the sense of comparison with barriers (or viscosity sense) we have that
\[ (m-\varrho^E) V_{p_{\infty}} \geq \left[ -m \nabla \phi - \vec{b}(m-\varrho^E) \right] \cdot \vec{\nu} \text{ on } \p \{p_{\infty}>0\}. \] 

Analogously, we have the following proposition:

\begin{proposition} Let the pair $(\rho,\phi)$ be a classical sub-solution on $B \times (t_1,t_2)$, then if
\begin{enumerate}
    \item $\varrho_{\infty}(t=t_1) \geq \rho(t=t_1)$ \text{ on } B 
    \item $p_{\infty} \geq \phi$ \text{ on } $\p B \times [t_1,t_2]$ ,
\end{enumerate} then we have that $\varrho_{\infty} \geq \rho$ on $B \times [t_1,t_2].$
\end{proposition}

So in the sense of comparison with barriers (or viscosity sense) we have that
\[ (m-\varrho^E) V_{p_{\infty}} \leq \left[ -m \nabla \phi - \vec{b}(m-\varrho^E) \right] \cdot \vec{\nu} \text{ on } \p \{p_{\infty}>0\} \]

Hence, we have derived the velocity law in the viscosity sense. That is we have proven Theorem $\ref{velocity-law-theorem}$. 

\section{Congestion Case} \label{Characterization}  In this section we focus on the congestion case. That is we assume throughout this section that \eqref{congestion} holds. So now if $\Om$ is a compact subset of $\R^d$ such that $p_k$ is uniformly supported in $\Om \times [0,T]$ then there exists by continuity an $\e>0$ such that
\begin{equation} m_t + \e \leq mf + \nabla \cdot(m \vec{b}) \text{ in } \Om \times [0,T]. \label{new_monotocity} \end{equation} 

First we recall the definition of the streamlines. Since the external density solves a transport equation, the natural coordinates are the streamlines (or characteristics). For any $x_0 \in \R^d$, we let the streamline $X(\cdot,x_0)$ be the unique solution to
\[ \begin{cases} \p_t X(t,x_0) = -\vec{b}(X(t,x_0)) \text{ for } t \in \R \\ X(0,x_0)=x_0 \end{cases}. \] As $\vec{b}$ is Lipschitz, we know that there is a unique solution for all time and is continuously differentiable in time. This uniqueness actually implies that for any $t_0 \in \R$ that $X(t_0,x) : \R^d \rightarrow \R^d$ is invertible and its inverse is given by $X(-t_0,x)$ (see \cite{kim2019singular}). \\

We will show that $\{p_{\infty}(t)>0\} := \Sigma(t)$ is non-decreasing with respect to the streamlines. That is if $x \in \Sigma(t)$, then $X(s, X(-t,x)) \in \Sigma(s)$ for all $s \geq t$ (the streamline containing $x$ at time $t$ is in $\Sigma(s)$ for all $s \geq  t$). This will follow from deriving a uniform positive retention principle for $p_k$ along streamlines through a refined Aronson-B\'enilan estimate. We refine the barrier used in Section \ref{general-semi-convexity}. \\  

To derive such an Aronson-B\'enilan estimate, we first drop the subscript $k$ and recall that $w := \frac{1}{m} \nabla \cdot(m \nabla p)$. Then we define the elliptic operator for sufficiently large $C>0$
\[ \mathcal{L}[\phi] := -C-C|\nabla p|^2 + 2\nabla \phi \cdot \nabla p + \nabla \phi \cdot \vec{b} + (k-1)( \phi(\phi+F)+(\nabla \phi + \nabla F) \cdot(p \nabla \lambda + 2 \nabla p) + p(\Delta \phi + \Delta F)). \] 
Then when $C$ is large enough, by using \eqref{w1}, \eqref{w2}, \eqref{w3} with a similar computation as in Section \ref{semi-convexity} we see that $ \p_t w \geq \mathcal{L}[w]$. \\ 

Now $\tau>0$ and for $\alpha,\beta>0$ we define our barrier
\[ \psi(x,t;\tau):=-F+\frac{\alpha p - \beta}{k-1} - \frac{1}{(k-1)(t-\tau)}.\] We note that our constants $C$ and $\tilde{C}$ in estimating $\p_t \psi$ and $\mathcal{L}[\psi]$ are also independent of $\alpha$ and $\beta$. Plugging in $\psi$ into $\mathcal{L}$ along with Young's Inequality on the $2\nabla \psi \cdot \nabla p$ term implies there exists a large enough constant $\tilde{C} \geq C $ such that
\[ \mathcal{L}[\psi] \geq -\tilde{C}-\tilde{C}|\nabla p|^2 + \frac{2 \alpha}{k-1} |\nabla p|^2 + \frac{\alpha}{k-1} \nabla p \cdot \vec{b} + \alpha pw + 2\alpha |\nabla p|^2 + (k-1)[\psi^2 + \psi F  ].  \] Then from using \eqref{PME_pressure}
\[ \p_t \psi \leq \tilde{C} + \frac{\alpha}{k-1} |\nabla p|^2 + \frac{\alpha}{k-1} (\nabla p \cdot \vec{b}) + \alpha pw + \alpha pF + \frac{1}{(k-1)(t-\tau)^2}. \] This implies that
\begin{equation} \p_t \psi - \mathcal{L}[\psi] \leq  2\tilde{C} + |\nabla p|^2(\tilde{C}-\frac{\alpha}{k-1} -2\alpha) + \alpha p F + \frac{1}{(k-1)(t-\tau)^2}  -(k-1)[\psi^2+\psi F]. \label{psi_eqn1} \end{equation} So we first we choose $\beta > 2\alpha \sup_{k>1} ||p_k||_{L^{\infty}(Q_T)}$, so that $\alpha p - \beta \leq 0$. This combined with $F>0$ gives 
\begin{equation} (\psi^2 + \psi F) = (\psi+\frac{F}{2})^2-\frac{F^2}{4} \geq (-\frac{F}{2} + \frac{\alpha p - \beta}{k-1})^2 + \frac{1}{(k-1)^2(t-\tau)^2} - \frac{F^2}{4} \label{squared_term_1} \end{equation}
\begin{equation} =  \left( \frac{\alpha p - \beta}{k-1} \right)^2 - F\frac{(\alpha p - \beta)}{k-1} + \frac{1}{(k-1)^2(t-\tau)^2}. \label{squared_term_2} \end{equation} So using \eqref{squared_term_1} and \eqref{squared_term_2} in \eqref{psi_eqn1}
\[ \p_t \psi - \mathcal{L}[\psi] \leq 2\tilde{C} + |\nabla p|^2(\tilde{C}-\frac{\alpha}{k-1}-2\alpha) + F (2\alpha p - \beta) - \frac{ (\alpha p - \beta)^2}{k-1} . \] Now because $2 \alpha p - \beta \leq 0$ due to our choice of $\beta$ and $F \geq \e > 0$ inside $\Om \times [0,T]$ we have that
\[ \p_t \psi - \mathcal{L}[\psi] \leq 2\tilde{C} + |\nabla p|^2(\tilde{C}-\frac{\alpha}{k-1}-2\alpha) + \e (2\alpha p - \beta) - \frac{(\alpha p - \beta)^2}{k-1} \text{ in } \Om \times [0,T] . \]   Now we let $\alpha = \tilde{C}/2$ and $\beta = 2 \alpha \sup_{k>1} ||p_k||_{L^{\infty}(Q_T)} + 2\tilde{C}/\e$, then we see that
\[ \p_t \psi - \mathcal{L}[\psi] \leq 0 \text{ in } \Om \times [0,T]. \] That is $\psi$ is a sub-solution of a parabolic equation where $w$ is a super solution of in $\Om \times [0,T]$.  And recalling that $p$ is compactly supported inside $\Om \times [0,T]$ gives us $w \equiv 0$ on $\p \Om \times [0,T]$. Also we have $\psi \leq 0$ on $\p \Om \times [0,T]$, so now as $\lim_{\tau \rightarrow 0+} \psi(x,0;\tau)=-\infty$, we deduce from the comparison principle that 
\[ \frac{1}{m} \nabla \cdot(m \nabla p) = w(x,t) \geq \psi(x,t;0) \geq -F  - \frac{\beta}{k-1} - \frac{1}{(k-1)t}. \] Hence, we have shown

\begin{lemma} [Refined  Aronson-B\'enilan Estimate] Assume we are in the congestion case $(\ref{congestion})$ with Assumptions $\ref{Assum_1}$, $\ref{Assum_2}$, and $\ref{hard_constraint_restriciton}$, then we have the following refined semi-convexity estimate for all $k>1$
\begin{equation} \frac{1}{m} \nabla \cdot(m \nabla p_k) \geq -F -\frac{\beta}{k-1} - \frac{1}{(k-1)t}  \label{wk_refined}\end{equation} for some constant $\beta>0$ independent of $k$ in the sense of distributions.
\end{lemma}

This  Aronson-B\'enilan Estimate implies the following estimates by using \eqref{wk_refined} in \eqref{PME_pressure}

\begin{corollary} \label{uniform_SL_estimates} Assume we are in the congestion case $(\ref{congestion})$ with Assumptions $\ref{Assum_1}$, $\ref{Assum_2}$, and $\ref{hard_constraint_restriciton}$, then for all $k>1$ we have for some constant $\beta>0$ independent of $k$ that
\[ \frac{d}{dt} p_k(X(t,x_0),t) \geq -\left( \beta + \frac{1}{t} \right) p_k(X(t,x_0),t) \] and
\[ \frac{d}{dt} v_k(X(t,x_0),t) \geq -\frac{1}{k-1} \left( \beta + \frac{1}{t} \right) v(X(t,x_0),t) \] in the sense of distributions.
\end{corollary} 

The normalized density estimate implies the following monotonicity property:
\begin{lemma} The limiting normalized density $v_{\infty}:=\varrho_{\infty}/m$ is non-decreasing along streamlines for $t>0$.
\end{lemma}
Note that as the streamlines $X$ are a lipschitz bijection of $Q_T$, $X$ maps sets of full measure to sets of full measure. This along with the pressure estimate in Corollary \ref{uniform_SL_estimates} combined with Grönwall's Inequality gives us:
 \begin{lemma} [Retention Along Streamlines For Positive Time] Assume that Assumptions $\ref{Assum_1}$, $\ref{Assum_2}$, and  $\ref{hard_constraint_restriciton}$ are met, then we have the following estimate for all  $k>1$ and $\tau>0$
\[  p_k(X(t,x_0),t) \geq p_k(X(\tau,x_0),\tau) \exp[-(\beta+\frac{1}{\tau})(t-\tau)] \] for some constant $\beta>0$ independent of $k$ for a.e. $(x_0,t), (x_0,\tau) \in Q_T$ such that $t \geq \tau$.  This implies that for a.e. $(x_0,t)$ and $(x_0,\tau) \in Q_T$ such that $t \geq \tau$ that we have the estimate
\[ p_{\infty}(X(t,x_0),t) \geq p_{\infty}(X(\tau,x_0),\tau) \exp[-(\beta+\frac{1}{\tau})(t-\tau)]. \] This implies that for a.e. $(x,t) \in Q_T$ with $t>0$ that $\Sigma(t)$ is non-decreasing with respect to streamlines.
\end{lemma}

To obtain the retention property down to $t=0$, we use the classical sub-solution $|\Pi|_{+} = \max(\Pi,0)$, as described in Section ($\ref{barrier_appendix}$), to the pressure equation for all $k>1$. This barrier was inspired by the one in \cite{kim2019singular}. 

\begin{lemma} [Retention Property Along Streamlines For $t=0$] Assume that $p_k(x,0) \rightarrow p_{\infty}(x,0)$ uniformly such that $\Sigma(0)$ is open, we are in the congestion case $(\ref{congestion})$, and Assumptions $\ref{Assum_1}$, $\ref{Assum_2}$, and  $\ref{hard_constraint_restriciton}$. Then if $p_{\infty}(x,0)>0$ we have that $p_{\infty}(X(t,x),t)>0$ for a.e. $(x,t) \in Q_T.$ That is for a.e. $(x,t) \in Q_T$ that $X(t,\Sigma(0)) \subset \Sigma(t) $.
\end{lemma}

\begin{proof} Let $x_0 \in \Sigma(0)$, then there is an $r_1>0$ such that $B_{r_1}(x_0) \subset \Sigma(0)$. By choosing $\gamma$ and $r$ sufficiently small in the barrier $|\Pi|_{+}$ constructed in Lemma $\ref{Big_Pi_Barrier}$, we can assume that for all $k$ sufficiently large that
\[ |\Pi|_{+}(x,0) \leq p_k(x,0). \] Hence, by the comparison principle, we conclude that for all large $k$ that
\[ |\Pi|_{+}(x,t) \leq p_k(x,t) \Rightarrow |\Pi|_{+}(x,t) \leq p_{\infty}(x,t) \text{ for a.e. } (x,t) \in Q_T . \] But notice $|\Pi(X(t,x_0),t)|_{+}=\gamma^2>0 \Rightarrow p_{\infty}(X(t,x_0),t) \geq \gamma^2 > 0$ for a.e. $(x_0,t) \in Q_T$. That is we have $X(t,\Sigma(0)) \subset \Sigma(t)$ for a.e. $(x,t) \in Q_T$.
\end{proof}

By combining these two lemmas, we have that
\begin{theorem} [$\Sigma(t)$ Increases Along Streamlines] Assume that  $p_k(x,0) \rightarrow p_{\infty}(x,0)$ uniformly such that $\Sigma(0)$ is open, we are in the congestion case $(\ref{congestion})$, and Assumptions $\ref{Assum_1}$, $\ref{Assum_2}$, and $\ref{hard_constraint_restriciton}$ are met. Then $\Sigma(t) = \{p_{\infty}(t)>0\}$ is non-decreasing with respect to streamlines for a.e. $(x,t) \in Q_T$. \label{inc_streamline}
\end{theorem}

Now we can prove Theorem $\ref{identification}$.

\begin{proof} As $\{p_{\infty}(x,t)>0\} \subset \{\varrho_{\infty}(x,t)=m(x,t)\}$, it suffices to show if $(x_0,t_0) \in \text{Int}(\{p_{\infty}(x,t)=0\})$ then $\varrho_{\infty}(x,t)=\varrho^E(x,t)$. So let $(x_0,t_0) \in  \text{Int}(\{p_{\infty}(x,t)=0\})$, then by Lemma $\ref{inc_streamline}$, we conclude that $X(t,X(-t_0,x_0)) \notin \Sigma(t)$ for $0 \leq t \leq t_0$ (the streamline containing $(x_0,t_0)$ is not in $\Sigma(t)$ for previous time). Combining this with $\varrho_{\infty}$ weakly solving
\[ \p_t \varrho_{\infty} = \nabla \cdot(\varrho_{\infty} \vec{b}) + f \varrho_{\infty} \text{ in } \text{Int}(\{p_{\infty}(x,t)=0\}) \] we conclude that $\varrho_{\infty}(x_0,t_0)=\varrho^E(x_0,t_0)$.
\end{proof}

 This shows Theorem $\ref{identification}$. Note that Theorem $\ref{identification}$ implies that patch solutions stay patch solutions for all times for congestion.

\section{Appendix: Barrier} \label{barrier_appendix}
In this section we will construct some barriers for \eqref{PME_pressure} that we used throughout the article. 
\subsection{Super Solutions For Radial \texorpdfstring{$m$}{m} or \texorpdfstring{$d=1$}{d=1}}  We first construct a super solution to the pressure equation \eqref{PME_pressure} for all $k>1$ when $m$ is radial or we are in one spatial dimension.

\begin{lemma} Assume that either $m$ is radial $(m(x,t)=m(r,t)$ where $r=|x|$), or we are working in one spatial dimension ($d=1$). Then if Assumption \ref{Assum_1} is met , then the radial function
\[ \varphi(r,t) := \frac{1}{d} \int_0^{r} \frac{r'}{m(r',t)} dr'  \] solves
\begin{equation}  \frac{1}{m} \nabla \cdot(m \nabla \varphi) =  1 \text{ in } \R^d \label{phi_m}  \end{equation} such that there exists $C_1,C_2,C_3>0$ with $C_1 r^2 \leq |\varphi(r,t)| \leq C_2  r^2$ and $|\nabla \varphi(r,t)| \leq C_3|r|.$
\label{pressure_ellipitic_barrier}
\end{lemma}

With the above $\varphi$ we can construct a viscosity super solution to the pressure equation \eqref{PME_pressure} for all $k>1$ when $d=1$ or $m$ is radial.

\begin{definition}[Viscosity Super-Solutions] We say that $\psi$ is a viscosity super solution to  \eqref{PME_pressure} if for all $\phi \in C^{\infty}(\R^d)$ that touch $\psi$ from below at $(x_0,t_0)$ then
\begin{equation} \p_t \phi \geq |\nabla \phi|^2 + \nabla \phi \cdot \vec{b} + (k-1) \phi( \frac{1}{m} \nabla \cdot(m \nabla \phi) + F) \text{ at the point } (x_0,t_0). \label{viscosity_diff_ineq} \end{equation}
\end{definition}

\begin{lemma}[Viscosity Super Solutions To The Pressure Equation for Radial \texorpdfstring{$m$}{m} or in 1D] Assume that $m$ is radial or we are in one spatial dimension with Assumption $\ref{Assum_1}$ met. Then let $\varphi(x,t)$ be as in Lemma $\ref{pressure_ellipitic_barrier}$. Now if we fix a $\gamma>1$  then there exists a large $\alpha>0$ and a smooth increasing $R(t)$ with $R(0)=\gamma$ such that
\[ Z(x,t) := \alpha|R(t)-\varphi(x,t)|_{+}  \] is viscosity super solution to \eqref{PME_pressure} for all $k>1$. \label{super_solution_barrier_radial}
\end{lemma}

\begin{proof} We note that in this lemma, it is important to keep track of the constants that are not $\alpha$. Let us denote $U(t) := \{x:Z(x,t)>0\}$. We first show that $Z$ is a classical super solution of \eqref{PME_pressure} in $\{Z>0\}$. First thanks to Lemma \ref{pressure_ellipitic_barrier}, we have that
\[ \frac{1}{m} \nabla \cdot(m \nabla Z) = -\alpha/m \text{ in } U(t). \] This implies if $\alpha$ is sufficiently large that
\[  \frac{1}{m} \nabla \cdot(m \nabla Z) + F \leq 0 \text{ in } U(t) \quad \forall t>0. \] So by \eqref{PME_pressure}, it suffices to check that
\[ \p_t Z \geq \alpha^2|\nabla \varphi|^2 - \alpha (\nabla \varphi \cdot \vec{b}) \text{ in } U(t) \] to show that $Z$ is a classical super solution in $\{Z>0\}$. Now by Young's Inequality with $M := ||\vec{b}|||_{\infty}^2/2$, it suffices to show that 
\begin{equation} \p_t Z \geq \alpha^2\frac{3}{2}|\nabla \varphi|^2 + M \text{ in }  U(t). \label{desired_phi} \end{equation} 

Now we compute bounds on $\p_t \varphi$ and $|\nabla \varphi|^2$. By Lemma \ref{pressure_ellipitic_barrier} we have that $| \nabla \varphi(r,t)|^2 \leq C_3^2 r^2$ with $C_1 r^2 \leq |\varphi(r,t)|,$ so by choosing $K_{\varphi} := C_1 C_3^2$, we have 
\begin{equation} |\nabla \varphi|^2 \leq K_{\varphi} |\varphi|. \label{phi_nabla_2} \end{equation}  
Now to bound $\p_t \varphi$, we observe from Lemma \eqref{pressure_ellipitic_barrier} that \[  \p_t \varphi(r,t) = -\frac{1}{d} \int_0^r \frac{m_t r'}{m^2} dr', \]  so we have the bound \begin{equation} |\p_t \varphi(r,t)| \leq C_4 r^2 \leq M_{\varphi} \varphi(r,t) \label{phi_time_bound} \end{equation} for some $C_4,M_{\varphi}>0$. \\

Now by \eqref{phi_nabla_2} and \eqref{phi_time_bound} to show \eqref{desired_phi} it suffices to choose an $R$ such that
\[ \p_t R \geq (\frac{3}{2} \alpha^2 K_{\varphi} + M_{\varphi})R(t) + M. \] In particular, we can choose
\[ R(t) = (\gamma+M) \exp(( \frac{3}{2} \alpha^2 K_{\varphi} + M_{\varphi})t)  - M, \] which satisfies all the desired conditions. Thus, $Z$ is a classical super solution inside $\{Z>0\}$. \\

Because $0$ is a solution of the pressure equation \eqref{PME_pressure} and $Z$ is a classical super solution in $U(t)$, it suffices to show the viscosity touching condition \eqref{viscosity_diff_ineq} on $\p \{Z>0\}$ to conclude $Z$ is a viscosity super solution.  \\

Let $\phi$ be a smooth test function that touches $Z$ from below at $(x_0,t_0) \in \p \{Z>0\}$. We may assume that we have $|\nabla \phi(x_0,t_0)| \neq 0$ as then \eqref{viscosity_diff_ineq} would be trivial because $\p_t \phi(x_0,t_0) \geq 0$. Indeed, observe
\[ Z(x_0,t_0+h) - Z(x_0,t_0) \geq \phi(x_0,t_0+h) - \phi(x_0,t_0), \] which implies $\p_t \phi(x_0,t_0) \geq 0.$ \\

So now assume that $\phi$ touches $Z$ from below at $(x_0,t_0) \in \p \{Z>0\}$ such that $|\nabla \phi(x_0,t_0)| > 0$. First observe from \eqref{desired_phi} that by undoing our Young's Inequality that
\begin{equation} \p_t Z \geq |\nabla Z|^2 + (|\nabla Z|) ||\vec{b}||_{\infty} \text{ in } \{Z>0\} \label{Z1}. \end{equation} Now let $V_{\phi}$ and $V_Z$ denote the normal velocity of the free boundary of $\p \{x: \phi(x,t)>0\}$ and $\p \{x:Z(x,t)>0\}$  respectively at $(x_0,t_0)$ then from \eqref{Z1}
\[ \frac{\p_t \phi(x_0,t_0)}{|\nabla \phi(x_0,t_0)|} = V_{\phi} \geq V_{Z} \geq |\nabla Z(x_0,t_0)| + ||\vec{b}||_{\infty}. \] Note that $\nabla Z(x_0,t_0)$ refers to the limit of the $\nabla Z(x,t)$ from inside $\{Z>0\}$ to the point $(x_0,t_0)$ and is non-zero because $|\nabla \phi(x_0,t_0)| \leq |\nabla Z(x_0,t_0)|$ because $\phi$ is touching $Z$ from below. Then this implies
\[ \p_t \psi(x_0,t_0) \geq  |\nabla \psi(x_0,t_0)|^2 + \nabla \psi(x_0,t_0) \cdot \vec{b}(x_0,t_0),  \] which allows us to conclude that $Z$ is a viscosity super solution.
\end{proof}

\begin{remark} Assume that we can construct a $u$ such that $\nabla \cdot(m \nabla u) \geq 1$ where $C_1 |x|^2 \leq u(x,t) \leq C_2 |x|^2$ and $|\nabla u(x,t)| \leq C |x|$ with $|\p_t u(x,t)| \leq C |x|^2 + C$. Then the proof of Lemma \eqref{super_solution_barrier_radial} implies that we can construct a viscosity super solution of \eqref{PME_pressure} for all $k>1$ that is uniformly compact support in $Q_T$.  \label{remark_1}
\end{remark}

\subsection{Super Solutions When \texorpdfstring{$\nabla m$}{m} Decays Fast Enough At \texorpdfstring{$\infty$}{o}} 

Now we use this for the case of $|\nabla m(x,t)| \leq \e/|x|$ for $|x| \geq R/2$ to construct a positive sub-solution of
 the uniformly elliptic operator $\nabla \cdot(m(x,t) \nabla u)$ in $\R^d$. We first consider the unique solution of
\begin{equation} \begin{cases} \nabla \cdot(m(x,t) \nabla [\phi(x,t)]) =1 \text{ in } B_{R}(0) \times [0,T] \\ \phi(x,t) = 1 \text{ on } \p B_{R}(0) \end{cases}. \label{ellipitic_pressure_L} \end{equation} Now we use  a priori  estimates to show that we may assume that $\phi(x,t) > 0$ in $B_R(0) \times [0,T]$. \\

By standard  a priori estimates for uniformly elliptic PDEs (see \cite{han2011elliptic, gilbarg2015elliptic}), we have that there is a constant $C=C(m,\nabla m,R,d)$ such that
\[ \sup_{(x,t) \in B_R(0) \times [0,T]}|\phi(x,t)| \leq C. \]  So we may assume that $\phi(x,t) > 0$ on $B_{R}(0) \times [0,T]$ since otherwise we may consider $\phi(x,t)+2C$, which solves the same PDE with constant boundary data. \\

Next we will obtain regularity of $\phi$ in time. First by global Schauder estimates (see \cite{gilbarg2015elliptic}), there is a $C=C(m,\nabla m,R,d)$ such that
\begin{equation} \sup_{(x,t) \in B_R(0) \times [0,T]} |\phi(x,t)| + \sup_{(x,t) \in B_R(0) \times [0,T]} \sum_{i=1}^{d} |\p_{x_i} \phi(x,t)| + \sup_{(x,t) \in B_R(0) \times [0,T]} \sum_{i,j=1}^{d} |\p^2_{x_ix_j} \phi(x,t)| \leq C. \label{phi-bounds}  \end{equation} Hence, by differentiating $(\ref{ellipitic_pressure_L})$ in time, we have that
\[ \begin{cases} \nabla \cdot(m_t \nabla \phi) + \nabla \cdot(m \nabla \phi_t) = 0 \text{ in } B_R(0) \times [0,T] \\ \phi_t(x,t)=0 \text{ on } \p B_R(0) \end{cases}. \] The first term, $\nabla \cdot(m_t \nabla \phi)$ is smooth in space and bounded in $\overline{B_R(0)} \times [0,T]$, so
by treating this as a known term, we can conclude again by standard a priori arguments for uniformly elliptic PDE that there is a $C=C(m,\nabla m,R,d)$ such that
\[ |\phi_t(x,t)| \leq C. \]

 Observe that for $|x| \geq R$ by Cauchy-Schwarz we have that
\[ \nabla \cdot \left[ m \nabla(|x|^2) \right] = 2dm + 2\nabla m \cdot \vec{x} \geq 2d \delta - 2\e \geq \delta. \]  Note that we used $m(x,t) \geq \delta > 0$, $|\nabla m(x,t)| \leq \e/|x|$ for $|x| \geq R/2$, and assumed $\e=\e(\delta)$ is sufficiently small. From  $(\ref{phi-bounds})$, we can choose a $C \geq 1/\delta$ so large such that for $w(x):=1+C(|x|^2-R^2)$ and if $x \in \p B_R(0)$, we have that the normal derivatives are ordered:
\[ \frac{\p w}{\p r}(x) >  \frac{\p \phi}{\p r}(x),  \] which implies
\[ \varphi(x,t) := \begin{cases} \phi(x,t) \text{ for } |x| \leq R \\ w(x) \text{ for } |x| \geq R \end{cases} \] is a viscosity solution of $\nabla \cdot( m \nabla u) \geq 1$ (since the maximum of two viscosity sub-solutions  is a viscosity sub-solution \cite{crandall1992user} and $\varphi$ cannot be touched from above on $\p B_R(0)$). Also notice that $\p_t \varphi = 0$ for $|x| \geq R$, and $|\p_t \varphi|=|\p_t \phi| \leq C$ in $B_R(0)$, so $\p_t \varphi$ is bounded.  \\

Now note that $\varphi$ satisfies all the requirements of Remark \ref{remark_1}, so we conclude

\begin{lemma} [Super Solutions To The Pressure Equation when $\nabla m$ decays fast enough at $\infty$] Assume that Assumption $\ref{Assum_1}$ is met and there is an $R>0$ such that $|\nabla m(x,t)| \leq \e/|x| $ for $|x| \geq R/2$ and $\e=\e(\delta)$ sufficiently small and non-negative. Then for any $\gamma>1$ we have that there is an $R$ with exponential growth such that $R(0)=\gamma$ with $\p_t R \geq 0$ and
\[ Z(x,t) := \alpha |R(t)-\varphi(x,t)|_{+} \] is a viscosity super solution to the pressure equation \eqref{PME_pressure} when  $\alpha>0$ is sufficiently large and $k>1$. And $Z$ is bounded in $L^{\infty}(Q_T)$ and is compactly supported for any fixed time. \label{super_solution_case_2}
\end{lemma}

\subsection{Subsolutions For Congestion Case}

Now we construct a barrier for the pressure equation when we are in the congestion case $(\ref{congestion})$. By a similar argument as in \cite{kim2019singular} Lemma A.4 we have the following:
\begin{lemma} Let $x_0 \in \R^d$, $L$ be the Lipschitz Constant of $\vec{b}$, and for sufficiently small $0<\gamma \ll r$, then define
\[ \Pi(x,t) := \gamma^2 - r^2 e^{2Lt}|x-X(t,x_0)|^2 \] Then if we are in the congestion case $(\ref{congestion})$ and we have Assumption $\ref{Assum_1}$ then $\Pi(x,t)$ is a classical sub-solution of the pressure equation \eqref{PME_pressure} on the set $\{\Pi>0\} \cap \{0 < t< T\}$ for all $k>1$.
\label{Big_Pi_Barrier}
\end{lemma}

\begin{lemma} Let $\Pi$ be as in Lemma $\ref{Big_Pi_Barrier}$. Then $|\Pi|_{+} := \max \{\Pi,0\}$ is a viscosity sub-solution to the pressure equation for all $k>1$.
\end{lemma}

\begin{proof} This follows from an analogous proof in Lemma \ref{super_solution_barrier_radial} using the normal velocity of $\p \{ x : \Pi(x,t)>0\}.$
\end{proof}

\section{Appendix: Weak Solutions of mPME} \label{weak_mPME}
In this appendix, we will sketch the proofs of properties of weak solutions to \eqref{PME_Density}. \\

Note that \eqref{PME_Density} is a parabolic equation that degenerates on the zero level set, so we still expect a comparison principle to hold for weak sub and super solutions. We sketch the proof below, which is based on the proof in \cite{kim2018regularity}.

\begin{theorem}[Comparison Principle for mPME] Let $\overline{u}$ and $\underline{u}$ respectively be a weak super and sub solution of \eqref{PME_Density} with initial datas $\overline{u}_0$ and $\underline{u}_0$
 respectively. Assume further that $\overline{u}_0 \geq \underline{u}_0$ almost everywhere, then we have $\overline{u} \geq \underline{u}$ almost everywhere.
\end{theorem} 

\begin{proof} The proof is based on a Hilbert duality argument. Let $w := \overline{u}-\underline{u}$, then we have for any non-negative test function $\varphi$ that
\[  0 \leq \iint_{Q_T} w \left( \p_t \varphi + A \nabla \cdot(m \nabla \varphi) +f\varphi -\nabla \varphi \cdot \vec{b}  \right), \] where  $0 \leq A := \frac{ \overline{v}^k - \underline{v}^k}{\overline{u}-\underline{u}}$ with $ \overline{v} := \overline{u}/m$ and $\underline{v} := \underline{u}/m$. Then if $\xi$ is a non-negative test function, then we want to solve the following degenerate backwards parabolic PDE
\begin{equation} \begin{cases} \p_t \varphi + A \nabla \cdot(m \nabla \varphi) + f \varphi - \nabla \varphi \cdot \vec{b} = \xi \text{ in } Q_T \\ \varphi(x,T)=0 \label{dual_back} \end{cases}. 
\end{equation}  Note that solutions of \eqref{dual_back} will be non-negative because 0 is a sub-solution of \eqref{dual_back}. Then if we could solve \eqref{dual_back}, then we would have
\[ \iint_{Q_T} w \xi \geq 0 \] for all non-negative test functions, which would imply $w \geq 0$ almost everywhere. \\ 

To solve \eqref{dual_back}, we have to approximate $A$ by a smooth approximation $A_{\e}$ where $\e \leq A_{\e} \leq \frac{1}{\e}$ with $\e > 0$. One can then solve \eqref{dual_back} with $A$ replaced by $A_{\e}$ by using uniform parabolic theory. Then by following similar arguments as in \cite{kim2018regularity} Theorem 3.4, one sees that as $\e \rightarrow 0$ we obtain that $ \iint_{Q_T} w \xi \geq 0$.
\end{proof}

Our comparison principle immediately implies that we have uniqueness of weak solutions of \eqref{PME_Density}. So we now focus on existence. \\

Our existence proof is based on the methods in \cite{vazquez2007porous}. We first fix a large spatial domain $\Om$ and construct weak solutions on a bounded domain. We let $\Om_T := \Om \times (0,T)$, then we want to construct a solution to
\begin{equation} \begin{cases}
	\p_t \varrho_k = \nabla \cdot( \varrho_k(\nabla p_k + \vec{b})) + f \varrho_k \text{ in } \Om_T  \\
	\varrho_k(x,0)= \varrho_k^0(x) \text{ in } \Om \\
	\varrho_k(x,t) = 0 \text{ on } \p \Om \times [0,T]
	\end{cases}, \label{bounded_PME}
\end{equation}

\noindent where the initial data $\varrho_k^0(x)$ is compactly supported inside $\Om$. 

\begin{definition}[Weak Solutions to mPME on $\Om_T$] We say that $\varrho_k$ is a weak solution to   \eqref{bounded_PME} with initial data $\varrho_k^0 \in L^{\infty}(\Om)$ that is compactly supported in $\Om$ if $\varrho_k \in L^2(\Om_T)$ and $v_k := \varrho_k/m$ satisfies $v_k^k \in H_0^1(\Om_T)$ such that for all test functions $\varphi \in C^{\infty}_c(\Omega \times [0,\infty))$
\begin{equation}
	 \iint_{\Om_T} -\varrho_k \p_t \varphi =  \int_{\Omega} \varphi(x,0) \varrho_k^0(x) + \iint_{\Om_T} -m\nabla \varphi \cdot \nabla v_k^k - \varrho_k(\nabla \varphi \cdot \vec{b}) + f \varphi \varrho_k.  \label{weak_sub_soln_bound} \end{equation} 
\end{definition}

To construct a weak solution, we solve the following approximate solution. For $n \in \N$, we consider the approximate problem 
\begin{equation} \begin{cases}
	\p_t \varrho^{(n)}_k = \nabla \cdot( \varrho_k^{(n)}(\nabla p_k^{(n)} + \vec{b})) + f \varrho_k^{(n)} \text{ in } \Om_T  \\
	\varrho^{(n)}_k(x,0)= \varrho_k^0(x) + \frac{m(x,0)}{n} \text{ in } \Om \\
	\varrho^{(n)}_k(x,t) = m(x,t)/n\text{ on } \p \Om \times [0,T]
	\end{cases} \label{bounded_PME1}.
\end{equation}
We can compare $\varrho_k^{(n)}$ to the sub-solution $\delta m(x,t)\exp(-\alpha t)$ for small enough $\delta>0$ and large enough $\alpha>0$, to conclude $\varrho_k^{(n)} \geq \e_n > 0$. Then this implies $\varrho_k^{(n)}$ solves a uniformly parabolic PDE, so by following the arguments in \cite{vazquez2007porous} Chapter 3, we know that $\varrho_k^{(n)}$ is a smooth and is a classical solution to \eqref{bounded_PME}. \\

Now from the comparison principle, we have that $\varrho_k^{(n)} \geq \varrho_k^{(n+1)}$, so there exists a limit $\varrho_k := \lim_{n \rightarrow \infty} \varrho_k^{(n)}$. By the monotone convergence theorem, we have that $\varrho_k^{(n)} \rightarrow \varrho_k$ in $L^p(\Om_T)$ for $p \in [1,\infty)$. So from the weak formulation of \eqref{bounded_PME}, it suffices to show that if $v_k^{(n)} := \varrho_k^{(n)}/m$, then $(v_k^{(n)})^k \rightharpoonup v_k^k$ in $H^1(\Om_T)$ where $v_k := \varrho_k/m$. In particular, it suffices to show that $||\nabla (v^{(n)}_k)^k||_{L^2(\Om_T)} \leq C$ where $C$ is independent of $n$. To simplify our notations, let us write $q_n := (v_k^{(n)})^k$ and $u_n := \varrho_k^{(n)}$. Then we can multiply \eqref{PME_Density} by $(q_n - \frac{1}{n^k})$ which is 0 on $\p \Om$  to see by integration that $|| \nabla q_n||_{L^2(\Om_T)} \leq C$. This concludes the existence of weak solutions of \eqref{bounded_PME}. \\


Notice in particular, that this shows we can approximate our weak solutions of \eqref{bounded_PME} with positive smooth solutions of  \eqref{bounded_PME}. Then by using our smooth approximations we have a comparison principle for viscosity super solutions and our weak solutions on $\Om_T$. Now to go from $\Om_T$ to $Q_T$, we use our barriers constructed in Section \ref{barrier_appendix}. These barriers show that if we choose $\Om$ sufficiently large, then our weak solution of \eqref{weak_sub_soln_bound} will be compactly supported inside $\Om$. Then we can extend our weak solution by defining to be zero outside of $\Om$, which gives a weak solution of \eqref{PME_Density}. \\

\bibliographystyle{amsplain}
\bibliography{HS_Bib}

\end{document}